\documentclass{article}

\usepackage{blindtext}
\usepackage{subfiles}

\usepackage{packages/commonpackage}
\usepackage{packages/package_for_list_of_content}
\usepackage{packages/macro}
\usepackage{xcolor}

\usepackage[a4paper,width=150mm,top=27mm,bottom=27mm]{geometry}

\newcommand{\EQN}[1]{\begin{equation*}\begin{split} #1 \end{split}\end{equation*}}

\def\diam{\mathrm{diam}}
\def\mods{\mathrm{mod\ }}
\def\mf{\mathfrak}
\newcommand{\Ext}{\mathcal E}
\newcommand{\K}{\mathcal K}
\newcommand{\EEE}{\mathcal{E}}
\newcommand{\aaa}{\mathfrak{a}}

\numberwithin{equation}{section}
\newtheorem{theorem}{Theorem}[section]

\newtheorem{proposition}[theorem]{Proposition}
\newtheorem{corollary}[theorem]{Corollary}
\newtheorem{lemma}[theorem]{Lemma}

\newtheorem{example*}{Example}

\newtheorem{remx}[theorem]{Remark}
\newenvironment{remark}
  {\pushQED{\qed}\remx}
  {\popQED\endremx}

\title{Sharp Strichartz estimates on the one-dimensional torus}
\date{}
\author{Yangkendi Deng\footnote{Department of Mathematics and Statistics,
Beijing Institute of Technology, Beijing, China}, Chenjie Fan\footnote{ State Key Laboratory of Mathematical Sciences,
Academy of Mathematics and Systems Science,
Chinese Academy of Sciences, Beijing, China. }, Shaoming Guo\footnote{Chern Institute of Mathematics, LPMC and
New Cornerstone Science Laboratory,
Nankai University, Tianjin 300071, P.R. China}, Zihua Guo\footnote{School of Mathematics,
Monash University, Australia}, Yongming Luo\footnote{Faculty of Computational Mathematics and Cybernetics,
Shenzhen MSU-BIT University, China}}

\begin{document}

\maketitle

\begin{abstract}
In this article, we prove the sharp Strichartz estimates for the Schr\"odinger equation on the one-dimensional torus.  More precisely, via obtaining a sharp weak $L^{6}$ estimate with no derivative loss, we obtain sharp $L^{6}$ Strichartz estimate,  and sharp no loss $L^{p}$ Strichartz estimate for any $2<p<6$.  The latter answers a question by Bourgain \cite{Bourgain93}.
\end{abstract}


\section{Introduction}
\subsection{Background and main results}
Let $\mathfrak{a}=(a_n)_{n\in \Z}$ be a complex-valued sequence. Denote 
\begin{equation*}
\mathcal{E} \mathfrak{a}(x, t):= \sum_{n\in \Z}
a_n e(nx+ n^2 t).
\end{equation*}
Here $e(x)=e^{2\pi i x}$ for $x\in \R$ and $\mathcal{E}$ is short for ``extension''. 
Let $S \subset \mathbb{Z}$ be a finite set; one further defines 
\begin{equation*}
\mathcal{E}_{S} \mathfrak{a}(x, t):= \sum_{n\in S}
a_n e(nx+ n^2 t).
\end{equation*}
We will view $\mathcal{E} \mathfrak{a}$, $\mathcal{E}_{S}\aaa$ naturally as functions on $\mathbb{T}^{2}$.\\

The main result of this paper is as follows. 
\begin{theorem}\label{thm:weak-l6}
Let $\mathfrak{a}=(a_n)_{n\in \Z}$ be a complex-valued sequence.
Let $S \subset \mathbb{Z}$ be a finite set. Then
\begin{equation}\label{eq:weak-l6}
\norm{
\mathcal{E}_S \mathfrak{a}
}_{L^{6,\infty}(\T^2)}
\lesim 
\|\mathfrak{a}\|_{\ell^2},
\end{equation}
where the implicit constant is absolute, and in particular, does not depend on $S$. 
\end{theorem}

As an immediate consequence of Theorem \ref{thm:weak-l6} we obtain the following Strichartz estimates in Corollary \ref{cor:strong-lp}. The estimate \eqref{eq:subcritical-lp} follows from interpolating \eqref{eq:weak-l6} with a trivial $L^2$ estimate; the estimate \eqref{eq:strong-l6} follows from \eqref{eq:weak-l6} and the trivial $L^{\infty}$ bound: 
\begin{equation*}
    \|
\mathcal{E}_S \mf{a}
    \|_{L^{\infty}}\le \sum_{n\in S} |a_n| \lesim (\# S)^{\frac12} \|\mf{a}\|_{\ell^2},
\end{equation*}
via a straightforward dyadic decomposition according to the size of $\mathcal{E}_S \mf{a}$. 
\begin{corollary}\label{cor:strong-lp}
Let $S \subset \mathbb{Z}$ be a finite set with $\# S\geq 2$. We have 
\begin{equation}\label{eq:strong-l6}
\|\EEE_{S} \aaa\|_{L^{6}(\T^2)}\leq c_{6}(\log \# S)^{\frac{1}{6}}\|\aaa\|_{\ell^{2}},
\end{equation}
and 
\begin{equation}\label{eq:subcritical-lp}
\|
\mathcal{E}_S \mathfrak{a}
\|_{L^{p}(\T^2)}\leq c_{p}\|\aaa\|_{\ell^{2}}, 2<p<6
\end{equation}
where $c_{p}$ is a finite constant depending only on $p\in (2, 6)$.
\end{corollary}
\begin{remark}
We note that estimate \eqref{eq:strong-l6} has been obtained independently,  also with assistance of AI,  by Li and Yung \cite{li2026improved}.
\end{remark}

Before  discussing the estimates in Corollary \ref{cor:strong-lp}, let us introduce some notations.  Let $d\ge 1$ and $\bfS$ be a subset of $\Z^d$. Let $p>2$. We follow Bourgain \cite{Bourgain93} and define $K_p(\bfS)$ to be the smallest constant satisfying the inequality 
\begin{equation}\label{eq:restriction-constant}
    \Norm{
    \sum_{\bfn\in \bfS} a_{\bfn} e(\bfx\cdot \bfn)
    }_{L^p(\T^d)}\le K_p(\bfS) (\sum_{\bfn\in \bfS} |a_{\bfn}|^2)^{\frac12}.
\end{equation}
Here $\bfx\in \T^d$. To simplify notations, let us write 
\begin{equation*}
    K_p(N):= 
    K_p(
    \{(n, n^2): |n|\le N\}
    ).
\end{equation*}
Bourgain \cite{Bourgain93} proved that 
\begin{equation}\label{eq:bourgain-divisor-bound}
    K_6(N)\le \exp\pnorm{
    c \frac{\log N}{\log\log N}
    }
\end{equation}
for some finite constant $c$; in particular, 
\begin{equation}\label{eq:bourgain-epsilon-bound}
    K_6(N)\lesim_{\epsilon} N^{\epsilon},
\end{equation}
for every $\epsilon>0$. Moreover, it is well-known in the literature (see for instance \cite[page 118]{Bourgain93}) that 
\begin{equation}\label{eq:logarithmic-lower-bound}
    K_6(N)\gtrsim (\log N)^{\frac{1}{6}},
\end{equation}
which can be seen already by taking $a_{\bfn}=1$ in \eqref{eq:restriction-constant}. Regarding $p<6$, in the same paper \cite{Bourgain93}, Bourgain raised the question of whether $K_p(N)$ is bounded for $p<6$.  \\

Note that a special case of \eqref{eq:strong-l6} says that 
\begin{equation}\label{eq:sharp-logarithmic-bound}
    K_6(N)\lesim (\log N)^{
    \frac16
    }, 
\end{equation}
which is sharp by \eqref{eq:logarithmic-lower-bound}. Moreover, the estimate \eqref{eq:subcritical-lp} in 
Corollary \ref{cor:strong-lp} implies that 
\begin{equation*}
K_p(N) \lesim_p 1,
\end{equation*}
giving an affirmative answer to the above-mentioned open problem of Bourgain \cite{Bourgain93}.  

\begin{remark}\label{rem:lambdap}
For $\bfS\subset \Z^d$, if $K_p(\bfS)<\infty$, then we say $\bfS$ is a $\Lambda(p)$-set. For $\Gamma=\{(n,n^2):n\in \Z\}\subset \Z^2$, then now we know its $\Lambda$-spectrum above 2 is exactly
\EQN{
\{p>2: \Gamma\in \Lambda(p)\}=(2,6).
}
This is quite different from the classical Rudin/Bourgain “true $\Lambda(p_0)$” examples. Rudin constructed,
for even $p_0$, sets which are $\Lambda(p_0)$ but fail to be $\Lambda(q)$ for every $q>p_0$. Bourgain extended this to
every real $p_0 > 2$. Thus those have a closed endpoint $\Lambda(p)$-spectrum: $(2,p_0]$. The lattice parabola $\Gamma$ appears to be the first natural, explicitly described arithmetic $\Lambda(p)$-set with a non-attained critical exponent.
\end{remark}

We expand the discussion on the case $p=6$. 
The factor $\log N/\log\log N$ in \eqref{eq:bourgain-divisor-bound} comes from divisor bounds in analytic number theory, that is, the number of divisors of an integer $N$ is bounded by 
\begin{equation*}
    \exp\pnorm{
    c \frac{\log N}{\log\log N}
    }
\end{equation*}
for some finite constant $c$. Later in the breakthrough work \cite{BD15}, Bourgain and Demeter obtained the sharp $\ell^2$-decoupling inequality for the parabola (indeed for paraboloids in all dimensions), and as a consequence, they provided another proof of \eqref{eq:bourgain-epsilon-bound}. Let us state the $\ell^2$-decoupling inequality by Bourgain and Demeter \cite{BD15}. For each interval $J\subset [0, 1]$ and $g: [0, 1]\to \C$, denote 
\begin{equation*}
    \mathcal{E}_J g(\bfx):= \int_J g(\xi) e(\xi x_1+ \xi^2 x_2) d\xi,
\end{equation*}
with $\bfx=(x_1, x_2)$. For $\delta\in 2^{-\N}$, let $D_p(\delta)$ be the smallest constant such that 
\begin{equation*}
    \|\mathcal{E}_{[0, 1]} g\|_{L^p(\R^2)}\le D_p(\delta) \pnorm{
    \sum_{J\subset [0, 1], \ell(J)=\delta}
    \|\mathcal{E}_J g\|_{L^p(\R^2)}^2
    }^{\frac12}
\end{equation*}
holds for all $g$, where the sum over $J$  on the right hand side runs over all dyadic intervals of length $\delta$. Bourgain and Demeter \cite{BD15} proved that 
\begin{equation}\label{eq:bourgain-demeter-decoupling}
    D_6(\delta)\lesim_{\epsilon} \pnorm{\frac{1}{\delta}}^{\epsilon}
\end{equation}
for every $\epsilon>0$. By the standard discretization and parabolic rescaling argument, \eqref{eq:bourgain-demeter-decoupling} implies \eqref{eq:bourgain-epsilon-bound}: one takes $g$ to approximate a linear combination of Dirac masses at the points $\xi=(n+N)/(2N)$, $|n|\le N$, which correspond, after an affine parabolic rescaling, to the lattice points $(n,n^2)$. The proof by Bourgain and Demeter \cite{BD15} relies on a classical bilinear argument and a novel induction-on-scales argument, and does not use any analytic number theory. More recently, Li \cite{li20212}, inspired by Pierce's exposition \cite{pierce2017vinogradov} of Wooley’s efficient congruencing (see for instance \cite{wooley2016cubic}), improved \eqref{eq:bourgain-demeter-decoupling} and obtained that 
\begin{equation*}
    D_6(\delta) \le \exp\pnorm{
    c \frac{
    \log \frac{1}{\delta}
    }{
    \log\log \frac{1}{\delta}
    }
    }
\end{equation*}
for some finite constant $c$. It is very surprising that Li's proof \cite{li20212} does not use any analytic number theory either but gives rise to the same bound \eqref{eq:bourgain-divisor-bound} obtained by Bourgain \cite{Bourgain93} using divisor bounds. \\

Another breakthrough occurred when Guth, Maldague and Wang \cite{GMW20} used the high-low frequency method and came up with a very different proof of the $\ell^2$-decoupling inequality by Bourgain and Demeter \cite{BD15}. Perhaps most importantly, Guth, Maldague and Wang \cite{GMW20} obtained that 
\begin{equation*}
    D_6(\delta)\lesim \pnorm{
    \log \frac{1}{\delta}
    }^c
\end{equation*}
 for some finite constant $c$, and this implies 
 \begin{equation*}
     K_6(N)\lesim (\log N)^c,
 \end{equation*}
 for some possibly different constant $c$. Guo, Li and Yung \cite{GLY21} worked in $p$-adic fields, slightly refined the argument of \cite{GMW20} and obtained that 
 \begin{equation}\label{eq:guo-li-yung-bound}
     K_6(N)\lesim_{\epsilon} (\log N)^{2+\epsilon},
 \end{equation}
 for every $\epsilon>0$. Guo, Li and Yung worked in $p$-adic fields mainly for technical reasons: Wave packets in $p$-adic fields can have compact supports in both the frequency variable and the physical variable. \\

In the current paper, the right hand side of \eqref{eq:guo-li-yung-bound} is further improved to the sharp bound $(\log N)^{1/6}$, see \eqref{eq:sharp-logarithmic-bound}. 
The proof of \eqref{eq:sharp-logarithmic-bound}, as well as the proofs for Corollary \ref{cor:strong-lp} and Theorem \ref{thm:weak-l6}, follows essentially the high-low frequency approach of Guth, Maldague and Wang \cite{GMW20}. However, in order for the proof to work, it is absolutely crucial to work in $p$-adic fields. In other words, our proof does not yield sharp bounds for $D_6(\delta)$. This is very interesting because Guo, Li and Yung \cite{GLY21} introduced the $p$-adic framework to the study of $K_6(N)$ purely for technical reasons, but this ``technical'' reason turns out to be crucial for the resolution of the entire problem. \\

Another significant feature of the estimate \eqref{eq:strong-l6} in Corollary \ref{cor:strong-lp} is that the bound on its right hand side depends only on the cardinality of the set $S$.  To explain a further corollary, let us introduce more notations. Let $S\subset \Z$ be a finite set. For an integer $k\ge 1$, consider the $k$-fold energy
\begin{equation*}
J_k(S):=
\#\set{
(n_1, \dots, n_{2k})\in S^{2k}:
\sum_{k'=1}^k (n_{k'})^m=
\sum_{k''=k+1}^{2k} (n_{k''})^m, m=1, 2
}.
\end{equation*}
As a direct consequence of the estimate \eqref{eq:strong-l6} in Corollary \ref{cor:strong-lp}, we obtain the following corollary. 
\begin{corollary}\label{cor:cubic-energy}
Let $S\subset \Z$ be a finite set with $\# S\ge 2$. Then 
\begin{equation*}
J_3(S) \lesim (\# S)^3 \log (\# S).
\end{equation*}
\end{corollary}

Note that by taking $n_{k'}=n_{k'+ k}$ for each $1\le k'\le k$, we already have the lower bound $J_3(S)\ge (\# S)^3$. The logarithmic factor $\log (\# S)$ is already necessary when $S$ is of the form $\{0, 1, \dots, N\}$, see for instance the lower bound \eqref{eq:logarithmic-lower-bound}. \\

Bourgain and Demeter \cite{BD15} proved that 
\begin{equation*}
J_3(S)\lesim_{\epsilon} (\# S)^{7/2+\epsilon}
\end{equation*}
for every $\epsilon>0$. Very recently, Cushman, Demeter and Wu \cite{CDW26} obtained that 
\begin{equation*}
J_3(S)\lesim_{\epsilon} (\# S)^{3+\epsilon}
\end{equation*}
for every $\epsilon>0$. Let us also emphasize here that the arguments in both \cite{BD15} and \cite{CDW26} work for arbitrary collections of points, not only lattice points as in Corollary \ref{cor:cubic-energy}. \\

We discuss another application of Theorem \ref{thm:weak-l6} to dispersive PDEs. 
Note that $\EEE\aaa$ is related to the one-dimensional periodic Schr\"odinger equation
\begin{equation}\label{eq:linear-schrodinger}
\begin{cases}
iu_t+\Delta u=0,\\
u(0,x)=f(x).
\end{cases}
\end{equation}
It is convenient to denote the solution to \eqref{eq:linear-schrodinger} as $e^{it\Delta}f$. The main estimates \eqref{eq:strong-l6} and \eqref{eq:subcritical-lp} can be rewritten in the following forms: 
\begin{equation}\label{eq:periodic-l6-strichartz}
\|e^{it\Delta}P_{S}f\|_{L^{6}_{x,t}(\T\times[0,1])}\lesssim (\log \# S)^{\frac{1}{6}}\|f\|_{L^{2}_x(\T)}
\end{equation}
and
\begin{equation}\label{eq:subcritical-strichartz}
\|e^{it\Delta}f\|_{L^{p}_{x,t}(\T\times[0,1])}\lesssim \|f\|_{L^{2}_x(\T)},
\end{equation}
where $P_S$ denotes the frequency projection to the set $S\subset \Z$.  The latter estimate \eqref{eq:subcritical-strichartz} is sharp and has immediate applications to nonlinear Schr\"odinger equations
\begin{equation}\label{eq:nls}
\begin{cases}
iu_{t}+\Delta u=\mu |u|^{p-2}u,\\
u(0,x)=f,
\end{cases}
\end{equation}
where $\mu=\pm 1$. Equation \eqref{eq:nls} is called focusing when $\mu=-1$ and defocusing when $\mu=1$.

\begin{corollary}\label{cor:subcritical-nls}
 The nonlinear Schr\"odinger (NLS) equation \eqref{eq:nls} on the one-dimensional torus is locally well-posed (LWP) in $L^2$ when $p<6$. 
\end{corollary}

Corollary \ref{cor:subcritical-nls} follows immediately from \eqref{eq:subcritical-strichartz} and 
 a classical contraction mapping argument due to Bourgain \cite{Bourgain93}. By conservation of mass, the local solution extends globally, yielding global well-posedness (GWP) in $L^2$. \\

Let us discuss the relevance of the estimate \eqref{eq:periodic-l6-strichartz}. In the context of discussing the nonlinear equation \eqref{eq:nls}, we often take $S=[-N,N]\cap \mathbb{Z}$, and now the estimate \eqref{eq:periodic-l6-strichartz} reads as 
\begin{equation}\label{eq:frequency-localized-strichartz}
\|e^{it\Delta}P_{\leq N}f\|_{L^{6}_{x,t}(\T\times[0,1])}\lesssim (\log N)^{\frac{1}{6}}\|f\|_{L^{2}_x(\T)},
\end{equation}
where $P_{\le N}$ denotes the frequency projection to the set $S=[-N,N]\cap \mathbb{Z}$.  Because of the logarithmic loss, the 
estimate \eqref{eq:frequency-localized-strichartz} cannot imply even local well-posedness for equation \eqref{eq:nls} with $L^2$ initial data when $p=6$, in which case the equation is usually called mass-critical. Indeed, the equation is  ill-posed for $L^2$ initial data, in the sense that the solution map fails to be uniformly continuous as a map from any $L^2$ ball centered at zero into $C([0,T];L^2)$ for every $T>0$. This is analogous to the result of Herr and Kwak \cite{HK26} for two-dimensional mass-critical NLS, and we omit the proof.\\

LWP for the one-dimensional mass-critical NLS is known for all $H^{s}$ data with $s>0$, thanks to Bourgain's estimate
\begin{equation}\label{eq:strichartz-epsilon-loss}
\|e^{it\Delta}P_{\leq N}f\|_{L^{6}_{x,t}(\T\times[0,1])}\lesssim_{\epsilon} N^{\epsilon}\|f\|_{L^{2}_x(\T)},
\end{equation}
which is another way of writing \eqref{eq:bourgain-epsilon-bound}. 
The higher-dimensional analogues of  \eqref{eq:strichartz-epsilon-loss} are also known, even for irrational tori, due to Bourgain and Demeter \cite{BD15}:
\begin{equation*}
\|e^{it\Delta_{\T^{d}}}P_{\leq N}f\|_{L_{x,t}^{p_{d}}(\T^{d}\times[0,1])}\lesssim_{\epsilon}N^{\epsilon}\|f\|_{L^{2}_x(\T^{d})}, 
\end{equation*}
where 
\begin{equation*}
p_{d}=\frac{2(d+2)}{d}, \  \ d\ge 1.
\end{equation*}
At the endpoint $p_d$, 
it seems reasonable to conjecture that
\begin{equation}\label{eq:endpoint-strichartz-conjecture}
\|e^{it\Delta_{\T^{d}}}P_{\leq N}f\|_{L_{x,t}^{p_{d}}(\T^{d}\times[0,1])}\lesssim (\log N )^{1/p_{d}}\|f\|_{L^{2}_x(\T^{d})}.
\end{equation}
The estimate \eqref{eq:endpoint-strichartz-conjecture} was first established in the case $d=2$ in the breakthrough work of Herr and Kwak \cite{HK24}. The current work resolves the $d=1$ case, building on earlier works  \cite{GMW20} and \cite{GLY21}. Let us remark here that \cite{HK24} uses the Szemer\'edi--Trotter theorem as a main ingredient, while the current work mainly uses  tools from harmonic analysis. \\

The estimate \eqref{eq:frequency-localized-strichartz}
is also closely related to GWP for one-dimensional mass-critical NLS for $H^{s} (s>0)$ initial data with small $L^{2}$ norm. Indeed, if we could prove the following variant of \eqref{eq:frequency-localized-strichartz}
\begin{equation}\label{eq:short-time-strichartz}
\|e^{it\Delta}P_{\leq N}f\|_{L^{6}_{x,t}(\T\times[0,(\log N)^{-1}])}\lesssim \|f\|_{L^{2}_x(\T)},
\end{equation}
then we could follow the framework of Herr and Kwak \cite{HK24} and show that for every $s>0$, the one-dimensional mass-critical NLS is globally well-posed for all initial data in $H^s$ whose $L^2$ norms are small enough. Indeed, Herr and Kwak \cite{HK24}  obtained the exact analogue of \eqref{eq:short-time-strichartz} for $d=2$, and established GWP for two-dimensional mass-critical NLS with initial data in $H^s(\T^2)$ for every $s>0$ whose $L^2$ norms are small enough. \footnote{In another breakthrough work \cite{HK26}, Herr and Kwak removed the smallness requirement for $L^2$ norms, and established GWP for all $H^s (s>0)$ data.} In their argument, the power $-1$ in $\log N$ is crucial as they need $(\log N)^{-1}$ to be a non-summable sequence.  \\

The method in the current work does not seem strong enough to establish \eqref{eq:short-time-strichartz}. A slightly weaker version was obtained in \cite{SY26}:
\begin{equation*}
\|e^{it\Delta}P_{\leq N}f\|_{L^{6}_{x,t}(\T\times[0,(\log N)^{-C}])}\lesssim \|f\|_{L^{2}_x(\T)},
\end{equation*}
for some finite constant $C$. 
Combined with the seminal I-method \cite{CKSTT02}, Skouloudis and Yu  \cite{SY26} established small-mass GWP in  $H^{s}(\T)$ for every $s>0$.  We refer to \cite{dodson2026global} and reference therein for GWP for $H^{s}$ initial data, for $s$ relatively large.  General $H^{s}$-GWP remains an interesting open problem.

\subsection{Strategy of the proof of the main theorem}

In the proof of Theorem \ref{thm:weak-l6} we will use a bilinear method similar to those used in \cite{GMW20} and \cite{GLY21}.

For $m\ge2$, define
\begin{equation*}
 \K(m):=\sup_{0<\#(\supp (\mathfrak{a}))\le m}
 \frac{\|\Ext \mathfrak{a}\|_{L^{6,\infty}(\T^2)}}{\|\mathfrak{a}\|_{\ell^2}},
 \qquad
 \mathcal{B}(m):=\sup_{(\mf{a}_1,\mf{a}_2)}
 \frac{\|\Ext \mf{a}_1\,\Ext \mf{a}_2\|_{L^{3,\infty}(\T^2)}}{\|\mf{a}_1\|_{\ell^2}\|\mf{a}_2\|_{\ell^2}},
\end{equation*}
where in the second supremum $\mf{a}_1, \mf{a}_2$ are nonzero finitely supported functions on
$\Z$ such that 
\begin{equation*}
\#(\supp (\mf{a}_1)\cup\supp (\mf{a}_2))\le m
\end{equation*}
 and that their supports lie in two distinct residue
classes modulo $3$. The use of residue classes mod $3$ is not important, and any other odd primes would work equally well.  \\

Theorem \ref{thm:weak-l6} follows immediately from the following two propositions. 
\begin{proposition}\label{prop:linear-to-bilinear}
For every $m\ge 100$, 
\begin{equation*}
\mathcal{K}(m)^2\lesim \mathcal{B}(m)+1.
\end{equation*}
\end{proposition}

\begin{proposition}\label{prop:bilinear-to-linear}
For every $H\ge 1$ and $m\ge 100$,  
\[
\mathcal{B}(m)^3
\lesim
H^2+H^{-6}\mathcal{K}(m)^6.
\]
\end{proposition}

Proposition  \ref{prop:linear-to-bilinear} will be proven in Section \ref{sec:linear-to-bilinear}, and Proposition \ref{prop:bilinear-to-linear}  will be proven in Section \ref{sec:bilinear-to-linear}. The proof of Proposition \ref{prop:linear-to-bilinear} is somewhat standard. The proof of Proposition \ref{prop:bilinear-to-linear} follows mainly the framework of Guth, Maldague and Wang \cite{GMW20}; a more detailed discussion on the proof will be deferred to the beginning of Section \ref{sec:bilinear-to-linear}.

\vspace{0.5cm}

\noindent {\bf Notation.} We collect some of the notations used throughout the paper. We also largely follow the notation used in \cite{GLY21}.

\begin{enumerate}
\item We write \(\|\cdot\|_{L^p}\) and 
\(\|\cdot\|_{\ell^p}\) for the usual Lebesgue norms. 
\item We use the notation
\(A\lesssim B\) if \(A\le CB\) for an absolute constant \(C>0\), and
\(A\lesssim_p B\) if the implicit constant may depend on \(p\). We write
\(A\sim B\) when both \(A\lesssim B\) and \(B\lesssim A\) hold. 
\item For a prime $p$, we use $\Q_p$ to denote the field of $p$-adic numbers. For two elements $a_1, a_2\in \Q_p$, we use $|a_1-a_2|_p$ to denote their $p$-adic distance. 

\item We use $A\gg B$ to mean that $A\ge C B$ for some large universal constant $C$. 

\item  When there is no confusion, we abbreviate $r\ \mods 3^{j}$ to $r(3^{j})$.
\end{enumerate}

\noindent {\bf AI usage.} The use of LLM models is central to the current work, many rounds of chats are involved, and we believe that the LLM models are impacted in particular by the work of \cite{HK24, SY26, GLY21, CDW26}, but it is fair to say the first proof is given by LLM model, and that proof clearly is building upon \cite{GLY21}. It is also fair to say that the first AI generated proof is not very readable from our perspective. We have communicated an AI generated version to several experts in the field. The current version is not an AI generated work, and efforts have been made to digest, verify, and organize the material.\\

\noindent {\bf Acknowledgement.}Y. D. was supported by the National Natural Science Foundation of China (Grant No.~12501117).  C. F. was partially supported by NSFC Grant 12471232, 12288201, CAS Project for Young Scientists in Basic Research YSBR-031,  and   National Key R \& D Program of China2025YFA1018500.  S. G. is partly supported by the Nankai Zhide Foundation, NSFC Grant No. 12426204, the New Cornerstone Science Foundation, and the Fundamental Research Funds for the Central Universities No. 100-63263092.  Z. G. is supported by ARC FT230100588 and DP260100485. The authors are grateful to Professor Ciprian Demeter for the precious comments and suggestions about adding Remark \ref{rem:lambdap}.





\section{Proof of Proposition \ref{prop:linear-to-bilinear}}\label{sec:linear-to-bilinear}

The proof of Proposition \ref{prop:linear-to-bilinear} is a refinement of the bilinearizing argument in \cite[Section 5]{GLY21}. The argument in \cite{GLY21} uses several triangle/Minkowski inequalities, and loses several powers of $\log$. In the current paper, we will exploit frequency orthogonality more efficiently (see \eqref{eq:extension-square-function} and \eqref{eq:weak-l3-summation}), and we will only lose an absolute constant. \\

Before we start the proof, we first briefly recall some standard facts in p-adic Fourier analysis, see \cite{Taibleson75}, and \cite{VVZ94} for more details.
\subsection{3-adic Fourier projections}

In the current paper, instead of working in general $\Q_p$, we will only use $p=3$. The only reason for choosing $p=3$ is that it is the smallest odd prime; any other odd primes would work equally well for the rest of the argument. \\

For a function $V(x)$ on $\mathbb{T}$, and for an integer $j\ge 0$, one defines
\begin{equation}\label{eq:triadic-projection}
P_{j}V(x)=\frac{1}{3^{j}}\sum_{l=0}^{3^{j}-1}V(x+\frac{l}{3^{j}}).
\end{equation}
It is also instructive to write \eqref{eq:triadic-projection} on the frequency side, and we obtain 
\begin{equation*}
\widehat{P_j V}(\xi)= \widehat{V}(\xi)\cdot \mathbbm{1}_{\{3^j| \xi\}}.
\end{equation*}
Notice that
\begin{equation*}
\xi\equiv \xi'\pmod{3^j}
\Longleftrightarrow
|\xi-\xi'|_3\leq 3^{-j}.
\end{equation*}
Thus, in the $3$-adic language, $P_{j}$ is precisely the Fourier
projection onto the $3$-adic ball of radius $3^{-j}$ centered at $0$. For further preliminaries on harmonic analysis on $\Q_p$ related to Strichartz estimates on the one-dimensional torus, we refer to \cite[Section 2]{GLY21}.\\

We have the following variant of the classical Littlewood--Paley inequality. 
\begin{lemma}\label{lem:weak-littlewood-paley}
For every $r>2$,  one has  
\begin{equation*}
\|f\|^{2}_{L^{r,\infty}}\lesssim_r \|P_{J}f\|^{2}_{L^{r,\infty}}+\sum_{0\leq j<J}\|(P_{j+1}f-P_{j}f)\|^{2}_{L^{r,\infty}},
\end{equation*}
where the implicit constant is uniform in positive integers $J$. 
\end{lemma}
\begin{proof}[Proof of Lemma \ref{lem:weak-littlewood-paley}]
The proof follows from interpolating strong $L^p$ estimates for 
Littlewood-Paley inequalities in the $3$-adic setting (see for instance Taibleson \cite{Taibleson75}), and the 
Minkowski inequality, which explains $r>2$; the interpolation applied here is a weak-type interpolation, see for instance Grafakos \cite[Theorem 1.4.19]{grafakos2008classical} or Bergh and L\"ofstr\"om \cite[Theorem 5.2.1]{bergh2012interpolation}. 

\end{proof}

\subsection{Reduction of the proof of Proposition \ref{prop:linear-to-bilinear}}
To simplify notation, we denote
\begin{equation*}
U(x,t):=\EEE \aaa(x,t), \quad A:=\|\aaa\|_{\ell^{2}}^{2}.
\end{equation*}

We will pick $J_{0}$ large enough, so that 
\begin{equation*}
3^{J_{0}}\gg \diam(\supp(\aaa)).
\end{equation*}
\\

We need to estimate $ \|U\|_{L^{6,\infty}}^{2}$,  uniformly in $J_{0}$.\\

 Noting that $\|U\|_{L^{6,\infty}}^{2}=\||U|^{2}\|_{L^{3, \infty}}$, and by applying Lemma \ref{lem:weak-littlewood-paley} to $|U|^{2}$, we obtain that 
\begin{equation}\label{eq:extension-square-function}
\||U|^2\|_{L^{3,\infty}}\lesim A+ 
\pnorm{
\sum_{j=0}^{J_0-1} 
\norm{
D_j}_{L^{3,\infty}}^2
}^{1/2}, 
\end{equation}
where
\begin{equation*}
D_{j}:=P_{j}(|U|^{2})-P_{j+1}(|U|^{2}),
\end{equation*}

We emphasize that  frequency projection $P_j$ acts only on the first variable of $(x, t)$. \\

What is left is to analyze $D_{j}$. To continue, we need to recall the following residue-tree structure.  Let $U_{j,r}, A_{j,r}$ be defined for every $j$ by 
\begin{equation*}
U_{j, r}:= 
\mathcal{E} \pnorm{
\mathfrak{a}\cdot \mathbbm{1}_{
\{
n\equiv r\ (3^j)
\}
}
}
,\qquad
A_{j, r}:= 
\sum_{n\equiv r\ (3^j)} |a_n|^2.
\end{equation*}
Note that for every $j$, one has 
\begin{equation*}
U=\sum_{r(3^{j})}U_{j,r}
\end{equation*}

A simple but important algebra fact is 
\begin{equation}\label{eq:projected-square}
P_{j}(|U|^{2})=\sum_{r(3^{j})}|U_{j, r}|^{2}.
\end{equation}
The proof of \eqref{eq:projected-square} is by direct expanding the squares.\\

For a node $(j,r)$ with $0\leq j<J_0$, its children are the three residue
classes modulo $3^{j+1}$ contained in the residue class $r$ modulo $3^j$. 

We denote the set of these children by
\begin{equation*}
\mathrm{Ch}(j,r)
:=
\left\{
r'\pmod{3^{j+1}}:
r'\equiv r\pmod{3^j}
\right\}.
\end{equation*}
With these notations,  we have
\begin{equation} 
U_{j,r}
=
\sum_{r'\in\mathrm{Ch}(j,r)}U_{j+1,r'},
\qquad
A_{j,r}
=
\sum_{r'\in\mathrm{Ch}(j,r)}A_{j+1,r'}.
\end{equation}

Let the cross term $Z_{j,r}$ be further defined as 
\begin{equation}\label{eq:cross-term-definition}
Z_{j, r}:= |U_{j, r}|^2-
\sum_{r'\in \mathrm{Ch}(j, r)}|U_{j+1, r'}|^2.
\end{equation}
One obtains the following algebraic identities. 
\begin{lemma}\label{lem:cross-term-identities}
Under the above notation, we have 
\begin{equation}\label{eq:martingale-cross-terms}
D_j=P_{j}|U|^{2}-P_{j+1}|U|^{2}=
\sum_{r\ \mods 3^j}
Z_{j, r}
\end{equation}
and 
\begin{equation}\label{eq:cross-term-expansion}
Z_{j,r}
=
\sum_{\substack{
r_1',r_2'\in\mathrm{Ch}(j,r)\\
r_1'\neq r_2'
}}
U_{j+1,r_1'}\overline{U_{j+1,r_2'}}.
\end{equation}
\end{lemma}
\begin{proof}[Proof of Lemma \ref{lem:cross-term-identities}]
The former identity \eqref{eq:martingale-cross-terms} follows from the definition of $D_j$ and the identity \eqref{eq:projected-square}. The latter identity \eqref{eq:cross-term-expansion} follows from expanding the first square in \eqref{eq:cross-term-definition}. 
\end{proof}

The proof of Proposition \ref{prop:linear-to-bilinear} is based on estimate \eqref{eq:extension-square-function}, with the help of the following two lemmas. Lemma \ref{lem:weak-l3-summation} is about frequency orthogonality of the terms on the right hand side of \eqref{eq:martingale-cross-terms}; the $\ell^{3/2}$ summation comes from interpolating $\ell^2(L^2)$ with $\ell^1(L^{\infty})$.  Lemma \ref{lem:bilinear-cross-term} is essentially a re-writing of the bilinear constant $\mathcal{B}(m)$; the relative order of the summation and the square root on the right hand side of \eqref{eq:bilinear-cross-term-bound} is not important as there are only finitely many terms in the summation (say bounded by $10$).  
\begin{lemma}\label{lem:weak-l3-summation}
 For every $0\leq j<J_0$, we have
\begin{equation*}
\|D_j\|_{L^{3,\infty}}
\lesssim
\left(
\sum_{r\ \mods 3^j}
\|Z_{j,r}\|_{L^{3,\infty}}^{3/2}
\right)^{2/3}.
\end{equation*}
\end{lemma}

\begin{lemma}\label{lem:bilinear-cross-term}
For every $0\leq j<J_0$ and every residue
class $r\pmod{3^j}$, we have
\begin{equation}\label{eq:bilinear-cross-term-bound}
\|Z_{j,r}\|_{L^{3,\infty}}
\lesim
\mathcal{B}(m)
\sum_{\substack{
r_1',r_2'\in\mathrm{Ch}(j,r)\\
r_1'\neq r_2'
}}
\sqrt{
A_{j+1,r_1'}A_{j+1,r_2'}
}.
\end{equation}
\end{lemma}
\begin{remark}
The proof of Lemma \ref{lem:weak-l3-summation} depends on interpolation based on the Fourier support condition of $Z_{j,r}$.  It is through Lemma \ref{lem:bilinear-cross-term},  that the bilinear estimate information enters the linear estimate.  Recall the whole point of Proposition \ref{prop:linear-to-bilinear} is to reduce a linear estimate into a bilinear estimate.
\end{remark}
Let us first finish the proof of Proposition \ref{prop:linear-to-bilinear}, 
 assuming Lemma \ref{lem:weak-l3-summation} and Lemma  \ref{lem:bilinear-cross-term}. 
We shall prove
\begin{equation*}
\||U|^2\|_{L^{3,\infty}}
\lesssim
\bigl(1+\mathcal{B}(m)\bigr)A.
\end{equation*}
Since $\||U|^2\|_{L^{3,\infty}}=\|U\|_{L^{6,\infty}}^2$, this immediately
implies 
\begin{equation*}
\mathcal{K}(m)^2\lesssim 1+\mathcal{B}(m),
\end{equation*}
and finishes the proof of  Proposition \ref{prop:linear-to-bilinear}. By \eqref{eq:extension-square-function}, we already have
\begin{equation*}
\||U|^2\|_{L^{3,\infty}}
\lesssim
A+
\left(
\sum_{j=0}^{J_0-1}
\|D_j\|_{L^{3,\infty}}^2
\right)^{1/2}.
\end{equation*}
Thus it remains to show
\begin{equation}\label{eq:square-function-sum}
\sum_{j=0}^{J_0-1}
\|D_j\|_{L^{3,\infty}}^2
\lesssim
\mathcal{B}(m)^2A^2.
\end{equation}
To proceed, we first point out that Lemma
\ref{lem:weak-l3-summation} gives
\begin{equation}\label{eq:weak-l3-summation}
\|D_j\|_{L^{3,\infty}}
\lesssim
\left(
\sum_{r\ \mods 3^j}
\|Z_{j,r}\|_{L^{3,\infty}}^{3/2}
\right)^{2/3}.
\end{equation}
Next, define
\[
\Delta_{j,r}
:=
\sum_{\substack{
r'_1,r'_2\in\mathrm{Ch}(j,r)\\
r'_1\neq r'_2
}}
A_{j+1,r'_1}A_{j+1,r'_2}=|A_{j, r}|^2-
\sum_{r'\in \mathrm{Ch}(j, r)}|A_{j+1, r'}|^2.
\]
Since each parent node has at most three children, Lemma \ref{lem:bilinear-cross-term} and Cauchy--Schwarz give
$$\|Z_{j,r}\|_{L^{3,\infty}}
\lesim
\mathcal{B}(m) \Delta_{j,r}^{1/2}. $$
Thus 
\begin{equation}\label{eq:martingale-delta-bound}
\|D_j\|_{L^{3,\infty}}^2
\lesssim
\mathcal{B}(m)^2
\left(
\sum_{r\ \mods 3^j}
\Delta_{j,r}^{3/4}
\right)^{4/3}.
\end{equation}
We claim that
\begin{equation}\label{eq:delta-mass-drop}
\Delta_{j,r}
\lesssim
A_{j,r}^{2/3}\Lambda_{j,r},
\end{equation}
where
$$\Lambda_{j,r}
:=
A_{j,r}^{4/3}
-
\sum_{r'\in\mathrm{Ch}(j,r)}
A_{j+1,r'}^{4/3}. $$
To see this, first write $a=A_{j,r}$ and
$x_i=A_{j+1,r_i'}$, where $r_i'\in\mathrm{Ch}(j,r)$. Thus
$a=\sum_i x_i$ and
\begin{equation*}
\Delta_{j,r}=a^2-\sum_i x_i^2.
\end{equation*}
 If $a=0$, there is nothing to prove.
Otherwise set $t_i=x_i/a$, so that $\sum_i t_i=1$. For $0\leq t\leq1$,
\[
1-t
\leq
3(1-t^{1/3}).
\]
Hence
\[
\begin{aligned}
\Delta_{j,r}
=
a^2\sum_i t_i(1-t_i)
\leq
3a^2\sum_i t_i(1-t_i^{1/3})
=
3a^2
\left(
1-\sum_i t_i^{4/3}
\right)
=
3A_{j,r}^{2/3}\Lambda_{j,r},
\end{aligned}
\]
which proves \eqref{eq:delta-mass-drop} and consequently
$\Delta_{j,r}^{3/4}
\lesssim
A_{j,r}^{1/2}\Lambda_{j,r}^{3/4}$.
Therefore, by H\"older's inequality and $\sum_r A_{j,r}^2\leq A^2$, we further obtain 
\begin{align*}
\left(
\sum_{r\ \mods 3^j}
\Delta_{j,r}^{3/4}
\right)^{4/3}
&\lesssim
\left(
\sum_{r\ \mods 3^j}A_{j,r}^2
\right)^{1/3}
\left(
\sum_{r\ \mods 3^j}\Lambda_{j,r}
\right) \\
&\lesssim A^{2/3} \left(\sum_{r\ \mods 3^j}\Lambda_{j,r}
\right).
\end{align*}
Finally, 
\[
\begin{aligned}
\sum_{j=0}^{J_0-1}
\sum_{r\ \mods 3^j}
\Lambda_{j,r}
&=
A^{4/3}
-
\sum_{r\ \mods 3^{J_0}}
A_{J_0,r}^{4/3}\leq
A^{4/3}.
\end{aligned}
\]
Combining the above estimates with \eqref{eq:martingale-delta-bound}, we obtain
\eqref{eq:square-function-sum}, and thus finish the proof of Proposition \ref{prop:linear-to-bilinear} by assuming Lemma \ref{lem:weak-l3-summation} and Lemma  \ref{lem:bilinear-cross-term}.

\subsection{Proof of Lemma \ref{lem:weak-l3-summation}}
To start, one records the following lemma.
\begin{lemma}
\label{lem:weak-l3-fourier-synthesis}
Let $\{\widetilde{P}_\nu\}_{\nu\in I}$ be a finite family of Fourier projections
on $\T^2$ whose frequency supports are pairwise disjoint. Suppose that
\begin{equation}\label{eq:projection-uniform-bound}
\sup_{\nu\in I}
\|\widetilde{P}_\nu\|_{L^\infty\to L^\infty}
\lesssim 1.
\end{equation}
If $\widetilde{P}_\nu f_\nu=f_\nu$ for every $\nu\in I$, then
\begin{equation*}
\Norm{
\sum_{\nu\in I}f_\nu
}_{L^{3,\infty}}
\lesim
\left(
\sum_{\nu\in I}
\|f_\nu\|_{L^{3,\infty}}^{3/2}
\right)^{2/3},
\end{equation*}
where the implicit constant depends only on the implicit constant in \eqref{eq:projection-uniform-bound}. 
\end{lemma}

\begin{proof}[Proof of Lemma \ref{lem:weak-l3-fourier-synthesis}]
Note that by Plancherel and the pairwise disjointness of the frequency supports, we have 
\begin{equation*}
\Norm{\sum_{\nu\in I}\widetilde{P}_\nu f_\nu}_{L^2}
\leq
\pnorm{
\sum_{\nu\in I}\|f_\nu\|_{L^2}^2
}^{1/2}.
\end{equation*}
On the other hand, by the triangle inequality and \eqref{eq:projection-uniform-bound}, we have 
\begin{equation*}
\Norm{\sum_{\nu\in I}\widetilde{P}_\nu f_\nu}_{L^\infty}
\lesssim
\sum_{\nu\in I}\|f_\nu\|_{L^\infty}.
\end{equation*}
The desired estimate follows from interpolation. The interpolation required here is a vector-valued version of a weak-type interpolation, see for instance Grafakos \cite[Theorem 1.4.19]{grafakos2008classical}, and we leave out the details for  this standard generalization. 
\end{proof}
Let us continue with the proof of Lemma \ref{lem:weak-l3-summation}. Recall from Lemma \ref{lem:cross-term-identities}  that
\begin{equation*}
D_{j}=\sum_{r\ \mods 3^j} Z_{j,r}.
\end{equation*}
Define $\Pi_{j,r}$ by the Fourier projection onto $\Xi_{j,r}$, where 
\begin{equation*}
\Xi_{j,r}
:=
\left\{
(\xi,\eta)\in\mathbb Z^2:
3^j\mid \xi,\quad
3^{j+1}\nmid \xi,\quad
3^{2j}\mid (\eta-2r\xi)
\right\}.
\end{equation*}
By \eqref{eq:cross-term-expansion} and the definition of $U_{j+1, r'}$, we have
\begin{equation*}
    \supp(\widehat{Z_{j,r}}) \subset \Xi_{j,r}.
\end{equation*}
Moreover, it is elementary to  check that the sets $\Xi_{j,r}$, $r\pmod{3^j}$, are pairwise disjoint for fixed $j$. Therefore, in order to apply Lemma \ref{lem:weak-l3-fourier-synthesis}, it remains to prove 
\begin{equation}\label{eq:cap-projection-bound}
\|\Pi_{j,r}\|_{L^\infty\to L^\infty} \lesssim 1.
\end{equation}
The key point here is that 
\(\Pi_{j,r}\) is the difference of the Fourier projection operators onto the two \textbf{subgroups}
$$
\left\{
(\xi,\eta)\in \Z^2:
3^j\mid\xi,\quad
3^{2j}\mid(\eta-2r\xi)
\right\}
$$
and
$$
\left\{
(\xi,\eta)\in \Z^2:
3^{j+1}\mid\xi,\quad
3^{2j}\mid(\eta-2r\xi)
\right\}.
$$
The desired bound \eqref{eq:cap-projection-bound} follows from the fact that 
the Fourier projection operator onto a subgroup is bounded from \(L^\infty\) to \(L^\infty\) (see Theorem 1.1 in \cite{BOS21} and Theorem 2.1.2, Lemma 2.1.3 in \cite{Rudin62}). 

We want to emphasize in this last step we use more than merely the support of $Z_{j,r}$ is disjoint.

\subsection{Proof of Lemma \ref{lem:bilinear-cross-term}}
The proof is in some sense a scaling argument.\\

Recall from Lemma \ref{lem:cross-term-identities} that
\[
Z_{j,r}
=
\sum_{\substack{
r_1',r_2'\in\mathrm{Ch}(j,r)\\
r_1'\neq r_2'
}}
U_{j+1,r_1'}\overline{U_{j+1,r_2'}}.
\]
Fix two distinct children $r_1',r_2'\in\mathrm{Ch}(j,r)$, and write
\[
r_i'\equiv r+3^j s_i\pmod{3^{j+1}},
\qquad
s_i\in\{0,1,2\},
\qquad i=1,2.
\]
 For $i=1,2$, define a sequence $\mathfrak{b}_i$ via its elements
\[
b_i(k):=a_{r+3^j k}\,
\mathbbm{1}_{\{k\equiv s_i\text{ mod } 3\}}.
\]
One can check that
\[
U_{j+1,r_i'}(x,t)
=
e(rx+r^2t)\,
\mathcal{E}\mathfrak{b}_i
\bigl(3^j(x+2rt),\,3^{2j}t\bigr),
\qquad i=1,2.
\]
Thus
\[
\|U_{j+1,r_1'}\overline{U_{j+1,r_2'}}\|_{L^{3,\infty}(\mathbb T^2)}
=
\|\mathcal{E}\mathfrak{b}_1
\overline{\mathcal{E}\mathfrak{b}_2}\|_{L^{3,\infty}(\mathbb T^2)}.
\]
The
definition of $\mathcal{B}(m)$ gives
\begin{align*}
\|U_{j+1,r_1'}\overline{U_{j+1,r_2'}}\|_{L^{3,\infty}}
&\leq
\mathcal{B}(m)
\|\mathfrak{b}_1\|_{\ell^2}
\|\mathfrak{b}_2\|_{\ell^2} \\
&= \mathcal{B}(m)
\sqrt{
A_{j+1,r_1'}A_{j+1,r_2'}
}.
\end{align*}
This completes the proof by the triangle inequality and by summing over the finitely many ordered pairs of distinct children.

\section{Proof of Proposition \ref{prop:bilinear-to-linear}}\label{sec:bilinear-to-linear}
In this section we give the proof of Proposition \ref{prop:bilinear-to-linear}. Indeed, we shall prove that the proposition follows from a single balanced level set estimate, stated in Proposition
\ref{prop:balanced-level-set}. Combining the unbalanced portion (see Sub-section \ref{sec:balanced-level-set} below), which is mainly controlled by the linear constant
$\mathcal{K}(m)$, we obtain the estimate 
\[
\mathcal{B}(m)^3
\lesim
H^2+H^{-6}\mathcal{K}(m)^6.
\]
Let us explain a few key words that appeared in the above short description. \\

First of all, the idea of using level sets to estimate the linear decoupling constant $\mathcal{K}(m)$ and the bi-linear decoupling constant $\mathcal{B}(m)$ can be traced back to the breakthrough work of Wolff \cite{MR1800068} where he came up with the idea of Fourier decoupling inequalities and proved the first sharp Fourier decoupling inequality for cones in $\R^3$. Similar ideas were also used by Bourgain (see \cite[Section 3]{Bourgain93}) in his earlier work on discrete Strichartz estimates on the tori, see also \cite{KV16} . In the current paper, when proving the level set estimate in Proposition \ref{prop:balanced-level-set}, we follow the framework of Guth, Maldague and Wang \cite{GMW20} using high-low frequency decompositions.

Secondly, for readers familiar with the broad-narrow framework of Bourgain and Guth \cite{BG11}, the word ``balanced" that appeared above means essentially the same as 
 ``broad", and ``unbalanced" means ``narrow". When working on the parabola, the contribution from narrow terms in \cite{BG11} can be controlled in an inductive manner, and the same idea works here as well, see Sub-section \ref{sec:balanced-level-set} for the details. \\

The main work of this section is the proof of Proposition \ref{prop:balanced-level-set}. As mentioned above, we mostly follow the high-low framework of Guth, Maldague and Wang \cite{GMW20}. Roughly speaking, when estimating the bilinear decoupling constant $\mathcal{B}(m)$, that is, when estimating the product term $\mathcal{E} \mathfrak{a}_1 \mathcal{E} \mathfrak{a}_2$ on $\T^2$, we will divide $\T^2$ into a low-frequency region $L$ and a sequence of high-frequency regions $\Omega_k$, with $k=1, 2, \dots$. The decomposition is achieved via introducing a crucial square function $G_k$ (defined in \eqref{eq:pruning-square-function} below): We will consider the Fourier support of $G_k$ and decompose it into two parts, a high-frequency part and a low-frequency part; depending on which parts dominate (pointwise), we will decompose the entire domain $\T^2$ into the disjoint union of $L$ and $\Omega_k$ for admissible $k$ (see \eqref{eq:high-region-definition} and \eqref{eq:low-region-definition} below). The contribution from the low-frequency region $L$ is elementary to handle; here we only comment on the high-frequency regions $\Omega_k$. The key idea in Guth, Maldague and Wang \cite{GMW20} is that the high-frequency Fourier support of the square function $G_k$ can be decomposed into disjoint regions, which allows us to apply Plancherel's theorem accordingly (continued with a Cordoba-Fefferman argument as explained in the proof of Lemma \ref{lem:key-square-function}) and pass to the square function $G_{k+1}$ at the next scale. We continue this process until we reach the smallest scales as required by decoupling inequalities, and for the scales where we terminate in between we apply bilinear Fourier restriction estimates as in Lemma \ref{lem:key-local-bilinear} to reach the required  smallest scales. \\

The general description above is still at the level of \cite{GMW20}. Let us now explain some key differences between the forthcoming argument and that in \cite{GMW20}. To estimate the left hand side of \eqref{eq:discrete-balanced-level-set}, the idea of Guth, Maldague and Wang \cite{GMW20} is to ``interpolate" an $L^4$ norm with an $L^{\infty}$ norm. Here the use of an $L^4$ norm is very natural, because of applying  Plancherel's theorem to square functions ($4=2+2$). However, to obtain a favorable $L^{\infty}$ estimate, they need to throw away wave packets with ``large" amplitude. If we choose the relevant $L^{\infty}$ norm appropriately, it is not difficult to imagine (for instance via Chebyshev's inequality) that those  wave packets with large amplitude only have minor contributions. As a consequence, we see that the ``pruned" function and the original function are different by only a negligible quantity (see Lemma \ref{lem:key-pruning}). Guth, Maldague and Wang \cite{GMW20} sacrificed a few powers of logarithms to throw away the wave packets with large amplitudes; however, as we will see later  in the definition of the square function $G_k$ in \eqref{eq:pruning-square-function} that we should not entirely throw away these ``bad" wave packets, but should include all their squared magnitudes in $G_k$; the advantage will become clear in Lemma \ref{lem:key-square-function} (and also the orthogonality relation \eqref{eq:square-function-martingale} as in the proof of Lemma \ref{lem:key-square-function}), and in particular we can keep an exact orthogonality relation and do not need to lose any logarithms. \\

The rest of the paper is organized as follows. The proof of Proposition
\ref{prop:balanced-level-set} is reduced to three key lemmas: Lemma
\ref{lem:key-pruning} provides the structural and approximation
properties of the pruning decomposition; Lemma \ref{lem:key-square-function}
gives the mean, variance, and global square-function bounds needed to
sum over all stopping regions, without a loss in the number of scales;
finally, Lemma \ref{lem:key-local-bilinear} supplies the local bilinear estimate.  Let us also emphasize that in the rest of the paper we will mostly cite 
\cite{GLY21}; however, we want to make it clear that whenever we cite \cite{GLY21}, we are also implicitly citing \cite{GMW20} as \cite{GLY21} generalizes \cite{GMW20} to the setting of $p$-adic fields. 


\subsection{Reduction to a key level set estimate}\label{sec:balanced-level-set}
Let $\mf{a}_1, \mf{a}_2$ be two sequences lying in two distinct residue classes modulo $3$, and write 
\begin{equation*}
A:= \|\mf{a}_1\|_{\ell^2}^2+ \|\mf{a}_2\|_{\ell^2}^2.
\end{equation*}
Choose $J_0\ge 2$ satisfying 
\begin{equation*}
    3^{J_0}> \diam(\supp(\mf{a}_1)\cup \supp(
\mf{a}_2
)).
\end{equation*}
For an integer $J_0^+\ge 2 J_0$, also denote $Q:= 3^{J_0^+}$ and 
\begin{equation*}
\Gamma_Q:= \{(i_1/Q, i_2/Q): 0\le i_1, i_2< Q\}\subset \T^2.
\end{equation*}
Hence $\Gamma_Q$ forms a grid of $Q^2$ points. 
For $H\ge 1$ and $\alpha>0$, consider the set 
\begin{equation*}
L_{\alpha,H}:=
\left\{
z\in\T^2:
\begin{aligned}
&\alpha\le |\mathcal{E}\mf{a}_1(z)\mathcal{E}\mf{a}_2(z)|^{1/2}<2\alpha,\\
&|\mathcal{E}\mf{a}_1(z)|^6+|\mathcal{E}\mf{a}_2(z)|^6
\le H^6|\mathcal{E}\mf{a}_1(z)\mathcal{E}\mf{a}_2(z)|^3
\end{aligned}
\right\}.
\end{equation*}
We point out that the set $L_{\alpha,H}$ is similar to the level sets defined in \cite[Proposition 6.3]{GLY21}, but 
the balanced condition
\[
|\mathcal{E}\mf{a}_1|^6+|\mathcal{E}\mf{a}_2|^6
\le H^6|\mathcal{E}\mf{a}_1\mathcal{E}\mf{a}_2|^3
\]
here is slightly easier and cleaner. 
On $L_{\alpha,H}$, this guarantees that both factors are bounded by
$O(H\alpha)$, which allows the pruning error to remain negligible at
the level of their product. The complementary unbalanced region is
controlled directly by the linear weak $L^6$ estimate. \\

We now give the crucial estimate associated with the level set $L_{\alpha,H}$.

\begin{proposition}\label{prop:balanced-level-set}
For every $H\ge 1, \alpha>0$ and $J_0^+\ge 2J_0$, we have 
\begin{equation}\label{eq:discrete-balanced-level-set}
\alpha^6 \frac{
\#(L_{\alpha, H}\cap \Gamma_Q)
}{Q^2}
\lesim 
H^2 A^3.
\end{equation}
\end{proposition}

With Proposition \ref{prop:balanced-level-set} in place, we give the proof of Proposition \ref{prop:bilinear-to-linear}.
\begin{proof}[Proof of Proposition \ref{prop:bilinear-to-linear}]
We first show that Proposition \ref{prop:balanced-level-set} implies the corresponding estimate on the continuous torus. For $\theta=(\theta_1,\theta_2)\in\T^2$, define
\begin{equation*}
a_i^\theta(n):=a_i(n)e(n\theta_1+n^2\theta_2),\qquad i=1,2.
\end{equation*}
Then
\begin{equation*}
\mathcal{E}\mf{a}_i^\theta(z)=\mathcal{E}\mf{a}_i(z+\theta).
\end{equation*}
Applying Proposition \ref{prop:balanced-level-set} to the shifted pair gives, for every $\theta\in\T^2$, 
\begin{equation*}
\alpha^6\frac{1}{Q^2}
\sum_{z\in\Gamma_Q}
\mathbbm{1}_{L_{\alpha,H}}(z+\theta)
\lesim H^2A^3.
\end{equation*}
Integrating in $\theta$ and using translation invariance of Haar measure, we obtain
\begin{equation*}
\frac{1}{Q^2}
\sum_{z\in\Gamma_Q}
\int_{\T^2}\mathbbm{1}_{L_{\alpha,H}}(z+\theta)\,d\theta
=|L_{\alpha,H}|,
\end{equation*}
which in turn yields
\begin{equation}\label{eq:continuous-balanced-level-set}
\alpha^6 |L_{\alpha,H}|
\lesim H^2A^3,
\qquad H\ge1,\ \alpha>0.
\end{equation}
We next bound $\mathcal{B}(m)$ in terms of $\mathcal{K}(m)$ via a general broad-narrow principle. Let $m$ be chosen such that 
\begin{equation*}
\#(
\supp(\mathfrak{a}_1) \cup 
\supp(\mathfrak{a}_2)
)\le m. 
\end{equation*}
By homogeneity, we may assume without loss of generality that
\begin{equation}\label{eq:equal-mass-normalization}
\|\mf{a}_1\|_{\ell^2}
=
\|\mf{a}_2\|_{\ell^2}.
\end{equation}
For $\alpha>0$, define the narrow superlevel set
\begin{equation*}
E_{\alpha,H}^{\mathrm{narrow}}
:=
\left\{
z\in\T^2:
\begin{aligned}
&|\mathcal{E}\mf{a}_1(z)\mathcal{E}\mf{a}_2(z)|^{1/2}>\alpha,\\
&|\mathcal{E}\mf{a}_1(z)|^6+|\mathcal{E}\mf{a}_2(z)|^6
>
H^6|\mathcal{E}\mf{a}_1(z)\mathcal{E}\mf{a}_2(z)|^3
\end{aligned}
\right\}.
\end{equation*}
Hence
\begin{equation*}
\begin{split}
E_{\alpha,H}^{\mathrm{narrow}}
\subset{}&
\set{
|\mathcal{E}\mf{a}_1|>2^{-1/6}H\alpha
}
\cup
\set{
|\mathcal{E}\mf{a}_2|>2^{-1/6}H\alpha
}.
\end{split}
\end{equation*}
By the definition of $\mathcal{K}(m)$ and \eqref{eq:equal-mass-normalization},
\begin{equation*}
\alpha^6|E_{\alpha,H}^{\mathrm{narrow}}|
\lesim
H^{-6}\mathcal{K}(m)^6A^3.
\end{equation*}
Note also that
\begin{equation*}
\set{
|\mathcal{E}\mf{a}_1\mathcal{E}\mf{a}_2|^{1/2}>\alpha
}
\subset
E_{\alpha,H}^{\mathrm{narrow}}
\bigcup
\pnorm{
\cup_{j\ge0}L_{2^j\alpha,H}
}.
\end{equation*}
By \eqref{eq:continuous-balanced-level-set},
\begin{equation*}
(2^j\alpha)^6|L_{2^j\alpha,H}|
\lesim H^2A^3.
\end{equation*}
Combining the above estimates, we obtain
\begin{equation*}
\begin{split}
&\alpha^6
\anorm{
\set{
|\mathcal{E}\mf{a}_1\mathcal{E}\mf{a}_2|^{1/2}>\alpha
}
}\\
&\qquad\lesim
H^{-6}\mathcal{K}(m)^6A^3
+
H^2A^3\sum_{j\ge0}2^{-6j}\\
&\qquad\lesim
A^3\pnorm{
H^2+H^{-6}\mathcal{K}(m)^6
}.
\end{split}
\end{equation*}
By the definition of the weak $L^3$ norm, we conclude that
\begin{equation*}
\mathcal{B}(m)^3
\lesim
H^2+H^{-6}\mathcal{K}(m)^6,
\qquad H\ge1.
\end{equation*}
This finishes the proof of Proposition \ref{prop:bilinear-to-linear}
\end{proof}


\subsection{Reduction to three Key Lemmas}\label{sec:three-key-lemmas}
The goal of this section is to reduce Proposition \ref{prop:balanced-level-set} to three key lemmas, which will be proven in several forthcoming sections. 

Recall the setup at the beginning of Section \ref{sec:balanced-level-set}. We also denote 
\begin{equation*}
U_1:= \mathcal{E} \mf{a}_1, \ \ U_2:= \mathcal{E} \mf{a}_2.
\end{equation*}
As Proposition \ref{prop:balanced-level-set} involves a discretized version of the level set estimate we need (say for instance \eqref{eq:continuous-balanced-level-set}), when proving this proposition, we will use discretized Fourier transforms. Let us be more precise. For $V: \Gamma_Q\to \C$, define 
\begin{equation*}
\widehat{V}(\xi, \eta):=
\frac{1}{Q^2}
\sum_{(x, t)\in \Gamma_Q}
V(x, t) e(-\xi x-\eta t), \ \ (\xi, \eta)\in (\Z/Q\Z)^2.
\end{equation*}
The Fourier inversion formula in this setting is 
\begin{equation*}
V(x, t)=
\sum_{\xi, \eta\ \mods Q}
\widehat{V}(\xi, \eta) e(\xi x+\eta t),
\end{equation*}
 Plancherel's theorem is in the form 
\begin{equation*}
\frac{1}{Q^2} 
\sum_{z\in \Gamma_Q} |V(z)|^2= \sum_{\xi, \eta\ \mods Q}|\widehat{V}(\xi, \eta)|^2,
\end{equation*}
and convolution is in the form
\begin{equation*}
\widehat{V_1 V_2}(\xi)=
\sum_{\zeta_1+\zeta_2=\xi} \widehat{V_1}(\zeta_1)\widehat{V_2}(\zeta_2).
\end{equation*}
Fix a residue $\rho\ \mods 3$, let $\mathcal{P}_{\rho} V$ be the frequency projection of $V$ to that residue class. \\

Fix an arbitrary $\lambda>0$\footnote{This is referred to as the \textit{pruning constant} in \cite{GLY21}}. Put 
\begin{equation*}
F_0:= U_1+ U_2.
\end{equation*}
For $k=1, \dots, J_0,$ we first apply a frequency decomposition (into $3$-adic balls) 
\begin{equation*}
h_{k, r}(x, t):= 
\sum_{\substack{
\xi, \eta\ \mods Q\\
\xi\equiv r (3^{J_0+1-k})
}}
\widehat{F}_{k-1}(\xi, \eta) e(\xi x+\eta t),
\end{equation*}
and then denote 
\begin{align}\label{eq:pruned-pieces}
F_{k, r}(z)&:= h_{k, r}(z) \mathbbm{1}_{\{
|h_{k, r}(z)|\le \lambda
\}
},\\
R_{k, r}(z)&:= 
h_{k, r}(z)-F_{k, r}(z),\label{eq:pruning-remainders}\\
F_k(z)&:= 
\sum_{r\ \mods 3^{J_0+1-k}} F_{k, r}(z),\notag\\
R_k(z)&:=
\sum_{r\ \mods 3^{J_0+1-k}} R_{k, r}(z).\notag
\end{align}
Next, we introduce square functions, and  do a high-low decomposition based on these square functions, in a way essentially the same as in \cite{GLY21}. Denote
\begin{equation}\label{eq:pruning-square-function}
G_{k}(z):= 
\sum_{r\ \mods 3^{J_0-k}}
|h_{k+1, r}(z)|^2+
\sum_{\ell=1}^k
\sum_{r\ \mods 3^{J_0+1-\ell}}
|R_{\ell, r}(z)|^2.
\end{equation}
Next, for $0\le k< J_0$, denote 
\begin{equation*}
\Gamma_{k, Q}:= (3^{-(J_0-k)}\Z)^2\cap \Gamma_Q.
\end{equation*}
A $\Gamma_{k, Q}$-block is a coset $z_0+ \Gamma_{k, Q}$; it is a $3$-adic ball of radius $3^{J_0-k}$ centered at $z_0$. \\


We decompose the level set $L_{\alpha, H}$ into high and low sets, similarly to \cite[Section 7.3]{GLY21}. Denote 
\begin{equation}\label{eq:high-region-definition}
\Omega_k:=\{
z\in \Gamma_Q: 
G_k(z)> 2A, G_{\ell}(z)\le 2A \text{ for } 0\le \ell< k
\}, \ 1\le k< J_0,
\end{equation}
and 
\begin{equation}\label{eq:low-region-definition}
L:= 
\{
z\in \Gamma_Q: G_{\ell}(z)\le 2A \text{ for } 0\le \ell< J_0
\}.
\end{equation}


\noindent We are now ready to state the three key lemmas for the proof of the level set estimate in Proposition \ref{prop:balanced-level-set}. They all hold for every $\lambda>0$. Recall that $U_1, U_2$ have frequencies supported in two distinct residue classes modulo $3$, and we pick $\rho_1\neq \rho_2$ and assume 
\begin{equation*}
\mathcal{P}_{\rho_i} U_i=U_i, \ \ i=1, 2.
\end{equation*}

\begin{lemma}[Key Lemma I]\label{lem:key-pruning}
 Under the above notation, we have 
\begin{enumerate}
\renewcommand{\labelenumi}{(\arabic{enumi})}
\item The sets $\Omega_1, \dots, \Omega_{J_0-1}, L$ partition $\Gamma_Q$.
\item Each $\Omega_k$ is a union of $\Gamma_{k, Q}$-blocks, and $L$ is a union of $\Gamma_{J_0-1, Q}$-blocks. 
\item For $i=1, 2$, we have 
\begin{equation}\label{eq:high-region-pruning-error}
|U_i(z)-
\mathcal{P}_{\rho_i} F_k(z)
|\le \frac{2A}{\lambda}, \ \ z\in \Omega_k,
\end{equation}
and 
\begin{equation}\label{eq:low-region-pruning-error}
|U_i(z)-
\mathcal{P}_{\rho_i} F_{J_0}(z)
|\le \frac{2A}{\lambda}, \ \ z\in L.
\end{equation}
\item We also have 
\begin{equation*}
|\mathcal{P}_{\rho_1} F_{J_0}(z)|^2+
|\mathcal{P}_{\rho_2} F_{J_0}(z)|^2
\le 
G_{J_0-1}(z)\le 2A
\end{equation*}
on $L$. 
\end{enumerate}
\end{lemma}


\begin{lemma}[Key Lemma II]\label{lem:key-square-function}
We have
\begin{enumerate}
\renewcommand{\labelenumi}{(\arabic{enumi})}
\item the mean identity 
\begin{equation}\label{eq:square-function-mean}
\frac{1}{Q^2}
\sum_{z\in \Gamma_Q}
G_k(z)=A, \ \ 0\le k< J_0,
\end{equation}
\item the variance control 
\begin{equation}\label{eq:square-function-variance}
\frac{1}{Q^2}
\sum_{z\in \Gamma_Q}
|G_{J_0-1}(z)-A|^2
\lesim 
\lambda^2 A,
\end{equation}
\item the estimate
\begin{equation}\label{eq:square-function-stopping-bound}
\frac{1}{Q^2}
\sum_{k=1}^{J_0-1}
\sum_{z\in \Omega_k} 
G_k(z)^2\lesim \lambda^2 A. 
\end{equation}
\end{enumerate}
\end{lemma}

\begin{remark}
We explain in a little more detail why this lemma is important. By \eqref{eq:square-function-mean}, we have
$
\frac{1}{Q^2}\sum_{z\in\Gamma_Q}|G_k(z)|^2 \ge A^2.$
In view of \eqref{eq:local-bilinear-square-function}, it is crucial to control the sum of the contributions from the different $G_k^2$. The key point is that we only need to control each $G_k$ on stopping region $\Omega_k$, as expressed in \eqref{eq:square-function-stopping-bound}. On this region, we have $G_k\le 2(G_k-A)$. Moreover,
$
Q^{-2}\sum_{z\in\Gamma_Q}|G_k(z)-A|^2
$
admits a better bound, by \eqref{eq:square-function-variance}. This is natural, since in general, for a random variable $X$, one expects better bound $\mathbb{E}(X-\mathbb{E}X)^{2}$,  compared with $\mathbb{E}(X^{2})$.
\end{remark}


\begin{lemma}[Key Lemma III]\label{lem:key-local-bilinear}
For $1\le k< J_0$ and every $\Gamma_{k, Q}$-block $B$, it holds that 
\begin{equation}\label{eq:local-bilinear-square-function}
\sum_{z\in B}
|\mathcal{P}_{\rho_1} F_k(z)
\mathcal{P}_{\rho_2} F_k(z)
|^2
\lesim
\sum_{z\in B}
G_k(z)^2.
\end{equation}
\end{lemma}


In the rest of this section, we will assume Lemmas \ref{lem:key-pruning}--\ref{lem:key-local-bilinear}, and finish the proof of Proposition \ref{prop:balanced-level-set}.

\begin{proof}[Proof of Proposition \ref{prop:balanced-level-set}]
 If
$0<\alpha\le \sqrt{2A}$, then
\[
\alpha^6 \frac{\#(L_{\alpha,H}\cap \Gamma_Q)}{Q^2}
\le 8A^3\lesim H^2A^3.
\]
Assume from now on that $\alpha>\sqrt{2A}$. Take
\begin{equation*}
\lambda:=\frac{100HA}{\alpha}.
\end{equation*}
Fix $z\in L_{\alpha,H}\cap\Omega_k$ with $1\le k<J_0$. By
\eqref{eq:high-region-pruning-error} and the choice of $\lambda$,
\[
\left|U_i(z)-\mathcal{P}_{\rho_i}F_k(z)\right|
\le \frac{2A}{\lambda}
=\frac{\alpha}{50H},
\qquad i=1,2.
\]
On $L_{\alpha,H}$ we also have
\[
|U_1(z)|,\ |U_2(z)|<2H\alpha.
\]
Combining these estimates, we obtain
\[
\left|
\mathcal{P}_{\rho_1}F_k(z)\mathcal{P}_{\rho_2}F_k(z)
-U_1(z)U_2(z)
\right|
<\frac{\alpha^2}{2}.
\]
Since $|U_1(z)U_2(z)|\ge\alpha^2$ on $L_{\alpha,H}$, the triangle
inequality therefore yields
\begin{equation}\label{eq:pruned-product-lower-bound}
\left|
\mathcal{P}_{\rho_1}F_k(z)\mathcal{P}_{\rho_2}F_k(z)
\right|
\ge \frac{\alpha^2}{2},
\qquad
z\in L_{\alpha,H}\cap\Omega_k,
\qquad
1\le k<J_0.
\end{equation}
The same argument, using \eqref{eq:low-region-pruning-error}, gives
\[
\left|
\mathcal{P}_{\rho_1}F_{J_0}(z)
\mathcal{P}_{\rho_2}F_{J_0}(z)
\right|
\ge \frac{\alpha^2}{2},
\qquad z\in L_{\alpha,H}\cap L.
\]
However, on $L$, (4) in Lemma \ref{lem:key-pruning} gives
\[
\left|
\mathcal{P}_{\rho_1}F_{J_0}(z)
\mathcal{P}_{\rho_2}F_{J_0}(z)
\right|
\le
\frac12\left(
|\mathcal{P}_{\rho_1}F_{J_0}(z)|^2
+
|\mathcal{P}_{\rho_2}F_{J_0}(z)|^2
\right)
\le A<\frac{\alpha^2}{2}.
\]
Thus the two bounds are incompatible unless
$L_{\alpha,H}\cap L=\varnothing$.\\

By (2) in Lemma \ref{lem:key-pruning}, each $\Omega_k$ is a union
of $\Gamma_{k,Q}$-blocks. Hence Lemma \ref{lem:key-local-bilinear} gives
\begin{equation}\label{eq:high-region-bilinear-bound}
\sum_{z\in\Omega_k}
\left|
\mathcal{P}_{\rho_1}F_k(z)\mathcal{P}_{\rho_2}F_k(z)
\right|^2
\lesim
\sum_{z\in\Omega_k}G_k(z)^2,
\qquad 1\le k<J_0.
\end{equation}
Using (1) in Lemma \ref{lem:key-pruning},
\eqref{eq:pruned-product-lower-bound}, \eqref{eq:high-region-bilinear-bound}, and
\eqref{eq:square-function-stopping-bound}, we obtain
\[
\begin{aligned}
\alpha^6\frac{\#(L_{\alpha,H}\cap\Gamma_Q)}{Q^2}
&=
\frac{\alpha^6}{Q^2}
\sum_{k=1}^{J_0-1}\#(L_{\alpha,H}\cap\Omega_k)\\
&\le
\frac{4\alpha^2}{Q^2}
\sum_{k=1}^{J_0-1}\sum_{z\in\Omega_k}
\left|
\mathcal{P}_{\rho_1}F_k(z)\mathcal{P}_{\rho_2}F_k(z)
\right|^2\\
&\lesim
\frac{\alpha^2}{Q^2}
\sum_{k=1}^{J_0-1}\sum_{z\in\Omega_k}G_k(z)^2\\
&\lesim
\alpha^2\lambda^2A
\lesim H^2A^3.
\end{aligned}
\]
This proves Proposition \ref{prop:balanced-level-set}.
\end{proof}

\section{Key Lemma I: Proof of Lemma \ref{lem:key-pruning}}\label{sec:key-pruning}
\subsection{Frequency supports}

We first show that when doing truncations like the one in \eqref{eq:pruned-pieces}, frequency supports of relevant functions do not change. 

For $1\le j\le J_0$ and $r\ (\mods 3^j)$, define 
\EQN{
\Theta_{j, r}:= 
\{
(\xi, \eta)\ \mods Q: \xi\equiv r\ (\mods 3^j), \ \eta\equiv \xi^2\ (\mods 3^{2j})
\}.
}
For these frequency caps, we have the nested property
\begin{equation*}
\Theta_{j+1, r'}\subset \Theta_{j, r},
\end{equation*}
whenever $r'\equiv r\ (\mods 3^j)$. 

\begin{lemma}[Local constancy property]\label{lem:local-constancy}
We have the following

\begin{enumerate}
\renewcommand{\labelenumi}{(\arabic{enumi})}
    \item $\supp(\widehat V)\subset \Theta_{j,r}$ if and only if
\[
V\left(x+\frac{1}{3^j},t\right)
=e\left(\frac{r}{3^j}\right)V(x,t)
\]
and
\[
V\left(x-\frac{2r}{3^{2j}},t+\frac{1}{3^{2j}}\right)
=e\left(-\frac{r^2}{3^{2j}}\right)V(x,t).
\]
Therefore, if $\supp(\widehat V)\subset \Theta_{j,r}$,
then $\supp(\widehat{V\varphi(|V|)})\subset \Theta_{j,r}$
for any function $\varphi$.

\item If $\supp(\widehat V)\subset \Theta_{j,r}$, then for all $(u,v)\in \Z^2$
\EQN{
V\pnorm{x+\frac{u}{3^j}, t+ \frac{v}{3^j}}=& e\pnorm{\frac{r u+ r^2 v}{3^j}} V(x, t).
}
Therefore, the function $|V|$ stays constant on each coset of $(3^{-j}\Z)^2$.
\end{enumerate}
\end{lemma}
\begin{proof}[Proof of Lemma \ref{lem:local-constancy}]

Fix $1\le j\le J_0$ and $r\ (\mods 3^j)$. Under the condition
$\xi\equiv r\ (\mods 3^j)$, the congruence
\[
\eta\equiv \xi^2\ (\mods 3^{2j})
\]
is equivalent to
\[
\eta-2r\xi+r^2\equiv 0\ (\mods 3^{2j}).
\]
We have
\EQN{
V\left(x+\frac{1}{3^j},t\right)=&\sum_{\xi,\eta}\widehat{V}(\xi,\eta)e(x\xi+t\eta)e\left(\frac{\xi}{3^j}\right)=e\left(\frac{r}{3^j}\right)V(x,t)\\
V\left(x-\frac{2r}{3^{2j}},t+\frac{1}{3^{2j}}\right)
=&\sum_{\xi,\eta}\widehat{V}(\xi,\eta)e(x\xi+t\eta)e\left(\frac{-2r \xi+\eta}{3^{2j}}\right)=e\left(-\frac{r^2}{3^{2j}}\right)V(x,t)
}
if and only if 
\EQN{
\widehat{V}(\xi,\eta)e\left(\frac{\xi}{3^j}\right)=&\widehat{V}(\xi,\eta)e\left(\frac{r}{3^j}\right)\\
\widehat{V}(\xi,\eta)e\left(\frac{-2r \xi+\eta}{3^{2j}}\right)=&\widehat{V}(\xi,\eta)e\left(-\frac{r^2}{3^{2j}}\right)
}
which is equivalent to say $\xi\equiv r\ (\mods 3^j)$ and $\eta-2r\xi+r^2\equiv 0\ (\mods 3^{2j})$
for $(\xi,\eta)\in \supp (\widehat{V})$. 
The phase factors above have modulus one, so $|V|$ is invariant under
both translations. Hence $V\varphi(|V|)$ satisfies the same
two translation identities as $V$, and its Fourier support is contained
in $\Theta_{j,r}$. This proves (1). 

For (2), note that for $(\xi,\eta)\in \Theta_{j, r}$ we have
\begin{equation*}
\xi\equiv r\ (\mods 3^j), \ \ \eta\equiv r^2\ (\mods 3^j).
\end{equation*}
By the Fourier inversion formula, we obtain 
\EQN{
V\pnorm{x+\frac{u}{3^j}, t+ \frac{v}{3^j}}=&
\sum_{\xi, \eta\ (\mods Q)}\widehat{V}(\xi, \eta)e(\xi x+\eta t)e\pnorm{\frac{\xi u+ \eta v}{3^j}}\\
=& e\pnorm{\frac{r u+ r^2 v}{3^j}} V(x, t).
}
This proves (2).  
\end{proof}

\begin{samepage}
\begin{lemma}\label{lem:pruning-structure}
Under the notations in the previous section, the following statements holds. 
\begin{enumerate}
\renewcommand{\labelenumi}{(\arabic{enumi})}
\item For $1\le k\le J_0$, 
it holds that 
\begin{equation*}
\supp(\widehat{h}_{k, r}
), \ \supp(\widehat{F}_{k, r}
), \ \supp(\widehat{R}_{k, r}
)\subset 
\Theta_{J_0+1-k, r}.
\end{equation*}
\item We have 
\begin{equation*}
F_{k-1}=
\sum_{
r\ \mods 3^{J_0+1-k}
}
h_{k, r}, 
\end{equation*}
and $|F_{k, r}|\le \lambda$. 
\item For $2\le k\le J_0$, we have 
\begin{equation*}
h_{k, r}=
\sum_{
\substack{
r'\ \mods 3^{J_0+2-k}\\
r'\equiv r (3^{J_0+1-k})
}
}
F_{k-1, r'},
\end{equation*}
and $|h_{k, r}|\le 3\lambda$. 
\end{enumerate}
\end{lemma}
\end{samepage}
\begin{proof}[Proof of Lemma \ref{lem:pruning-structure}]



(1) Since the Fourier support of
$F_0=U_1+U_2$ lies on the parabola $\eta=\xi^2$, we have
\[
\supp(\widehat{h}_{1,r}),
\ \supp(\widehat{F}_{1,r}),
\ \supp(\widehat{R}_{1,r})
\subset \Theta_{J_0,r}.
\]
Inductively, if the required support property holds at level $k-1$, then
decomposing $F_{k-1}$ according to its residue classes modulo
$3^{J_0+1-k}$ and using the nesting of the frequency caps gives
\[
\supp(\widehat{h}_{k,r})
\subset \Theta_{J_0+1-k,r}.
\]
Lemma \ref{lem:local-constancy} then gives the same inclusion for
$\widehat{F}_{k,r}$ and $\widehat{R}_{k,r}$. 

(2) By the definition of $h_{k,r}$, the classes
$r\pmod{3^{J_0+1-k}}$ partition the spatial frequencies. Hence
\[
F_{k-1}
=
\sum_{r\ (\mods 3^{J_0+1-k})}h_{k,r}.
\]
Moreover, \eqref{eq:pruned-pieces} immediately gives
\[
|F_{k,r}|\le \lambda.
\]

(3) By (1), for each fixed
$r\ (\mods 3^{J_0+1-k})$ we have
\[
h_{k,r}
=
\sum_{\substack{
r'\ (\mods 3^{J_0+2-k})\\
r'\equiv r\ (3^{J_0+1-k})
}}
F_{k-1,r'}.
\]
There are exactly three residue classes $r'$ in this sum. Since (2)
gives $|F_{k-1,r'}|\le\lambda$, it follows that
\[
|h_{k,r}|\le 3\lambda,
\qquad 2\le k\le J_0.
\]
\end{proof}

Before proceeding, let us explain briefly the motivation for computing the following identities \eqref{eq:pruning-projection-energy} to \eqref{eq:pruning-energy-conservation}. At each pruning step, the $L^2$ mass lost from $F_{k-1}$ is exactly transferred to the remainder $R_{k,r}$; after telescoping over the scales, the surviving mass together with all discarded masses is
exactly the original mass $A$. This will later give the mean identity for $G_k$.

By Lemma \ref{lem:pruning-structure} and Plancherel's theorem, we obtain 
\begin{equation}\label{eq:pruning-projection-energy}
\frac{1}{Q^2}
\sum_{z\in \Gamma_Q}\sum_{
r\ \mods 3^{J_0+1-k}
}
|h_{k, r}(z)|^2=
\frac{1}{Q^2}
\sum_{z\in \Gamma_Q}
|F_{k-1}(z)|^2.
\end{equation}
By the definition \eqref{eq:pruning-remainders} and by disjointness in the physical variables, it holds
\begin{equation}\label{eq:discarded-pruning-energy}
\frac{1}{Q^2}
\sum_{z\in \Gamma_Q}\sum_{
r\ \mods 3^{J_0+1-k}
}
|R_{k, r}(z)|^2=
\frac{1}{Q^2}
\sum_{z\in \Gamma_Q}
\pnorm{
|F_{k-1}(z)|^2
-
|F_k(z)|^2
}.
\end{equation}
Observe that 
\begin{equation}\label{eq:initial-square-mass}
\sum_{
r\ \mods 3^{J_0}
}
|h_{1, r}(z)|^2
=A, \ \forall z\in \Gamma_Q,
\end{equation}
and therefore 
\begin{equation}
\frac{1}{Q^2}
\sum_{z\in \Gamma_Q}|F_0(z)|^2=A.
\end{equation}
We sum over $k$ on both sides of \eqref{eq:discarded-pruning-energy}, and obtain 
\begin{equation}\label{eq:pruning-energy-conservation}
\frac{1}{Q^2} \sum_{z\in \Gamma_Q}\pnorm{
|F_k(z)|^2+\sum_{\ell=1}^k
\sum_r
|R_{\ell, r}(z)|^2
}=A,
\end{equation}
for $k\le J_0$.




By Lemma \ref{lem:local-constancy}, we obtain that 
\begin{equation*}
|h_{\ell, r}|^2, |F_{\ell, r}|^2, |R_{\ell, r}|^2 
\end{equation*}
are constant on every $\Gamma_{\ell-1, Q}$-block. By the nested properties of the blocks, we easily see that $G_k$ is constant on every $\Gamma_{k, Q}$-block, for every $0\le k< J_0$.

\subsection{Completion of the proof}
\begin{proof}[Proof of Lemma \ref{lem:key-pruning}]
(1) Fix $z\in \Gamma_Q$. If $G_{\ell}(z)\le 2A$ for every $0\le \ell< J_0$, then $z\in L$. Otherwise the set
\begin{equation*}
\{
\ell: 0\le \ell< J_0, G_{\ell}(z)> 2A
\}
\end{equation*}
has a least element, which we call $k$. By \eqref{eq:initial-square-mass} and the definition of $G_0$, we see that 
\begin{equation}\label{eq:initial-square-function}
G_0(z)=A, \ \forall z\in \Gamma_Q,
\end{equation}
which implies that $k\ge 1$. By the minimality of $k$, we know that 
\begin{equation}\label{eq:pre-stopping-square-function}
G_{\ell}(z)\le 2A, \ 0\le \ell< k,
\end{equation}
and therefore $z\in \Omega_k$. This finishes the proof of (1).\\

(2) We only show that $\Omega_k$ is a union of $\Gamma_{k, Q}$-blocks; the proof for $L$ is similar. By Lemma \ref{lem:local-constancy} we have
\begin{equation*}
G_{\ell}(z+w)=G_{\ell}(z), \ \forall z,
\end{equation*}
whenever $w\in \Gamma_{k, Q}$ and $0\le \ell\le k$, and therefore $z\in \Omega_k $ if and only if $z+w\in \Omega_k$. 

(3) To bound the approximation errors in \eqref{eq:high-region-pruning-error} and \eqref{eq:low-region-pruning-error}, we first write 
\begin{equation*}
\mathcal{P}_{\rho_i} F_0-
\mathcal{P}_{\rho_i} F_k
=
\sum_{\ell=1}^k
\sum_{
\substack{
r\ \mods 3^{J_0+1-\ell}\\
r\equiv \rho_i\ \mods 3
}
}
R_{\ell, r}.
\end{equation*}
Note that 
\begin{equation*}
|R_{\ell, r}(z)|\le \lambda^{-1} |R_{\ell, r}(z)|^2,
\end{equation*}
and therefore 
\begin{equation}\label{eq:pruning-remainder-error}
|U_i(z)-
\mathcal{P}_{\rho_i} F_k(z)
|
\le 
\frac{1}{\lambda}
\sum_{\ell=1}^k
\sum_{
r\ \mods 3^{J_0+1-\ell} 
}
|R_{\ell, r}|^2.
\end{equation}
Recall the definition of the square function $G_{k-1}$ that 
\begin{equation}\label{eq:previous-square-function}
G_{k-1}(z)=
\sum_{r\ \mods 3^{J_0-k+1}}
|
h_{k, r}(z)
|^2+
\sum_{\ell=1}^{k-1}
\sum_{
r\ \mods 3^{J_0+1-\ell} 
}
|R_{\ell, r}(z)|^2.
\end{equation}
By disjointness in the physical variables, we have
\begin{equation}\label{eq:pruning-pointwise-energy}
|h_{k, r}(z)|^2=
|F_{k, r}(z)|^2+
|R_{k, r}(z)|^2.
\end{equation}
We substitute the last two identities into \eqref{eq:pruning-remainder-error} and obtain 
\begin{equation}\label{eq:pruning-square-function-error}
|U_i(z)-
\mathcal{P}_{\rho_i} F_k(z)
|\le 
\frac{
G_{k-1}(z)
}{\lambda}.
\end{equation}
If $z\in \Omega_k$, then by \eqref{eq:pre-stopping-square-function}, we know in particular that 
\begin{equation*}
G_{k-1}(z)\le 2A.
\end{equation*}
This, combined with \eqref{eq:pruning-square-function-error}, proves \eqref{eq:high-region-pruning-error}. The proof for \eqref{eq:low-region-pruning-error} is similar, and thus is omitted. 

(4) First, note that 
\begin{equation*}
|
\mathcal{P}_{\rho_1} F_{J_0}(z)
|^2
+
|
\mathcal{P}_{\rho_2} F_{J_0}(z)
|^2
\le 
\sum_{r\ (\mods 3)} |F_{J_0, r}(z)|^2.
\end{equation*}
By \eqref{eq:previous-square-function} and \eqref{eq:pruning-pointwise-energy}, this is 
\begin{equation*}
\le G_{J_0-1}(z)\le 2A,
\end{equation*}
whenever $z\in L$.
\end{proof}

\section{Key Lemma II: Proof of Lemma \ref{lem:key-square-function}}\label{sec:key-square-function}

In this section we give the proof of Lemma \ref{lem:key-square-function}. The mean identity \eqref{eq:square-function-mean} follows directly from \eqref{eq:pruning-projection-energy} and \eqref{eq:pruning-energy-conservation}. Let us also remark here that \eqref{eq:initial-square-mass} tells us that 
\begin{equation*}
G_0(z)=A, \ \forall z\in \Gamma_Q,
\end{equation*}
which was observed in \eqref{eq:initial-square-function}. In the rest of this section, we will prove \eqref{eq:square-function-variance} and \eqref{eq:square-function-stopping-bound}. 


To simplify our notations, we introduce the block average: Using the language of $3$-adic analysis, It is just an  average over $3$-adic balls. For $0\le k< J_0$ and a function $V: \Gamma_Q\to \C$, define 
\begin{equation*}
(\E_k V)(z):=
\frac{1}{
\#(\Gamma_{k, Q})
}
\sum_{w\in \Gamma_{k, Q}}V(z+w).
\end{equation*}
Let us also collect a few simple properties of the average: For $0\le i\le k< J_0$, we have 
\begin{equation}\label{eq:conditional-average-tower}
\E_i\E_k=\E_i, \ \ \E_k\E_i= \E_i;
\end{equation}
if $S\subset \Gamma_Q$ is a union of $\Gamma_{k, Q}$-blocks, then 
\begin{equation}\label{eq:conditional-average-properties}
\sum_{z\in S} \E_k V(z)=
\sum_{z\in S} V(z), \ \ |\E_k V(z)|^2\le \E_k(|V|^2)(z),
\end{equation}
where the last inequality follows from Cauchy-Schwarz. 

\subsection{The low lemma and Cordoba-Fefferman argument}

Fix $1\le k< J_0$ and put $j=J_0-k$. By (3) in Lemma \ref{lem:pruning-structure}, we obtain 
\begin{equation*}
h_{k+1, r}= 
\sum_{
\substack{
b\ (\mods 3^{j+1})\\
b\equiv r\ (3^j)
}
}
F_{k, b}.
\end{equation*}
\begin{lemma}\label{lem:low-frequency-identity}
Under the above notation, we have 
\begin{equation}\label{eq:low-frequency-identity}
\E_{k-1}\pnorm{
\sum_{
r\ (\mods 3^j)
}
|h_{k+1, r}|^2
}
=
\sum_{
b\ (\mods 3^{j+1})
}
|F_{k, b}|^2.
\end{equation}
\end{lemma}
\begin{proof}[Proof of Lemma \ref{lem:low-frequency-identity}]
By (1) in Lemma \ref{lem:pruning-structure},
\[
\supp(\widehat{F}_{k,b})\subset \Theta_{j+1,b}.
\]
Hence, if $b\neq b'\ (\mods 3^{j+1})$, the spatial frequencies of
$F_{k,b}\widebar{F_{k,b'}}$ are nonzero modulo $3^{j+1}$, and therefore
\[
\E_{k-1}\left(F_{k,b}\widebar{F_{k,b'}}\right)=0.
\]
Moreover, $|F_{k,b}|$ is constant on each $\Gamma_{k-1,Q}$-block, so
\[
\E_{k-1}|F_{k,b}|^2=|F_{k,b}|^2.
\]

Using
\[
h_{k+1,r}
=
\sum_{\substack{
b\ (\mods 3^{j+1})\\
b\equiv r\ (3^j)
}}
F_{k,b},
\]
we may expand the square and obtain
\[
\E_{k-1}|h_{k+1,r}|^2
=
\sum_{\substack{
b\ (\mods 3^{j+1})\\
b\equiv r\ (3^j)
}}
|F_{k,b}|^2.
\]
Summing over $r\ (\mods 3^j)$ gives \eqref{eq:low-frequency-identity}.
\end{proof}

By the definition of $G_k$ and \eqref{eq:low-frequency-identity}, we obtain 
\begin{equation*}
\E_{k-1} G_k=\sum_{
r\ (\mods 3^{j+1})
} |F_{k, r}|^2
+
\sum_{\ell=1}^k \sum_{
r\ (\mods 3^{J_0+1-\ell})
}
|R_{\ell, r}|^2
\end{equation*}
where we also used the local constancy property of $|R_{\ell, r}|$. Using the pointwise disjointness in \eqref{eq:pruning-pointwise-energy}, we further obtain
\begin{equation*}
\E_{k-1} G_k=
\sum_{
r\ (\mods 3^{j+1})
} |h_{k, r}|^2
+
\sum_{\ell=1}^{k-1} \sum_{
r\ (\mods 3^{J_0+1-\ell})
}
|R_{\ell, r}|^2=
G_{k-1}.
\end{equation*}
This, combined with \eqref{eq:conditional-average-tower} and local constancy properties of $G_k$, implies that 
\begin{equation}\label{eq:square-function-martingale}
\E_i G_k=G_i, \ \ 0\le i\le k< J_0.
\end{equation}
Let us denote 
\begin{equation}\label{eq:square-function-increments}
D_k:= 
G_k-G_{k-1}=
\sum_{
r\ (\mods 3^{J_0-k})
} |h_{k+1, r}|^2
-
\sum_{
b\ (\mods 3^{J_0+1-k})
}
|F_{k, b}|^2,
\end{equation}
where in the last equality we used again \eqref{eq:pruning-pointwise-energy}. Note that each $D_k$ is constant on $\Gamma_{k, Q}$-blocks, and satisfies 
\begin{equation}\label{eq:increment-conditional-mean}
\E_{k-1} D_k=0.
\end{equation}
If $1\le i< k< J_0$, then by the fact that $D_i$ is constant on each $\Gamma_{k-1, Q}$-block $B$, and by \eqref{eq:increment-conditional-mean}, we obtain the orthogonality relation 
\begin{equation*}
\sum_{z\in B} D_i(z) D_k(z)= D_i(z_B) \sum_{z\in B} D_k(z)=0,
\end{equation*}
for any $z_B\in B$. Recall that $G_0=A$. By expanding squares and by the above orthogonality relation, we obtain 
\begin{equation}\label{eq:variance-orthogonality}
\frac{1}{Q^2}
\sum_{z\in \Gamma_Q}
|G_{J_0-1}(z)-A|^2=
\frac{1}{Q^2}
\sum_{k=1}^{J_0-1}
\sum_{z\in \Gamma_Q}|D_k(z)|^2.
\end{equation}
To bound the right hand side, we need the following simple lemma. 
\begin{lemma}\label{lem:fourier-resonance}
Let $1\le j\le J_0$, and suppose that $V_1, V_2, V_3, V_4$ satisfy $\supp \widehat{V_i}\subset \Theta_{j, r_i}$. Then 
\begin{equation*}
\sum_{z\in \Gamma_Q} V_1(z)V_2(z)\widebar{V_3(z)}
\widebar{V_4(z)}=0
\end{equation*}
unless, modulo $3^j$, either $(r_1=r_3 \text{ and } r_2=r_4)$ or $(r_1=r_4 \text{ and } r_2=r_3)$. 
\end{lemma}
\begin{proof}[Proof of Lemma \ref{lem:fourier-resonance}]
Suppose that
\[
\sum_{z\in \Gamma_Q}
V_1(z)V_2(z)\widebar{V_3(z)}\widebar{V_4(z)}
\neq 0.
\]
Since $3^{2j}\mid Q$ and
$\supp(\widehat{V_i})\subset\Theta_{j,r_i}$, Fourier orthogonality gives
$(\xi_i,\eta_i)\in\supp(\widehat{V_i})$ such that
\[
\xi_1+\xi_2\equiv \xi_3+\xi_4\ (\mods 3^{2j}),
\qquad
\xi_1^2+\xi_2^2\equiv \xi_3^2+\xi_4^2\ (\mods 3^{2j}).
\]
Since $2$ is invertible modulo $3^{2j}$, the two congruences imply
\[
(\xi_1-\xi_3)(\xi_1-\xi_4)\equiv 0\ (\mods 3^{2j}).
\]
It follows that either
\[
\xi_1\equiv \xi_3\ (\mods 3^{j})
\qquad\text{or}\qquad
\xi_1\equiv \xi_4\ (\mods 3^{j}).
\]
Combining this with $\xi_i\equiv r_i\ (\mods 3^j)$ and
\[
\xi_1+\xi_2\equiv\xi_3+\xi_4\ (\mods 3^j),
\]
we conclude that, modulo $3^j$, either
\[
r_1=r_3,\qquad r_2=r_4,
\]
or
\[
r_1=r_4,\qquad r_2=r_3.
\]
This proves the lemma.
\end{proof}

We continue to bound the right hand side of \eqref{eq:variance-orthogonality}. We use the expression \eqref{eq:square-function-increments} for $D_k$. By (3) in Lemma \ref{lem:pruning-structure} and by expanding squares, we obtain 
\begin{equation}\label{eq:increment-cross-terms}
D_k=
\sum_{
r\ (\mods 3^{J_0-k})
}
\sum_{\substack{
b, b'\ (\mods 3^{J_0+1-k})\\
b\equiv b'\equiv r\ (3^{J_0-k}), \ b\neq b'
}
}
F_{k, b}\widebar{F_{k, b'}}.
\end{equation}
We substitute \eqref{eq:increment-cross-terms} to \eqref{eq:variance-orthogonality}, apply Lemma \ref{lem:fourier-resonance}, and obtain 
\begin{equation}\label{eq:increment-energy}
\sum_{z\in \Gamma_Q} |D_k(z)|^2=
\sum_{
r\ (\mods 3^{J_0-k})
}
\sum_{\substack{
b, b'\ (\mods 3^{J_0+1-k})\\
b\equiv b'\equiv r\ (3^{J_0-k}), \ b\neq b'
}
}
\sum_{z\in \Gamma_Q}
|F_{k, b}(z)F_{k, b'}(z)|^2.
\end{equation}
To bound the right hand side of \eqref{eq:increment-energy}, we apply Lemma \ref{lem:fourier-resonance} (the Cordoba-Fefferman argument) again to 
\begin{equation*}
h_{k+1, r}= 
\sum_{
\substack{
b\ (\mods 3^{J_0-k+1})\\
b\equiv r\ (3^{J_0-k})
}
}
F_{k, b},
\end{equation*}
and obtain 
\begin{equation*}
\begin{split}
\sum_{z\in \Gamma_Q}
|h_{k+1,r}(z)|^4
=
& 
\sum_{
\substack{
b\ (\mods 3^{J_0-k+1})\\
b\equiv r (3^{J_0-k})
}
}
\sum_{z\in \Gamma_Q} |F_{k, b}(z)|^4\\
&+ 
2
\sum_{\substack{
b, b'\ (\mods 3^{J_0+1-k})\\
b\equiv b'\equiv r\ (3^{J_0-k}), \ b\neq b'
}
}
\sum_{z\in \Gamma_Q}
|F_{k, b}(z)F_{k, b'}(z)|^2.
\end{split}
\end{equation*}
This, combined with \eqref{eq:increment-energy}, implies that 
\begin{equation*}
\frac{1}{Q^2}\sum_{z\in \Gamma_Q} |D_k(z)|^2
=\frac{1}{2Q^2}\sum_{z\in \Gamma_Q}
\Bigg(
\sum_{r\ (\mods 3^{J_0-k})}|h_{k+1,r}(z)|^4
-\sum_{b\ (\mods 3^{J_0+1-k})}|F_{k,b}(z)|^4
\Bigg).
\end{equation*}
We substitute this back to \eqref{eq:variance-orthogonality}, sum over $k$ directly, and obtain 
\begin{equation}\label{eq:quartic-telescoping}
\begin{aligned}
& \frac{1}{Q^2} \sum_{z \in \Gamma_Q}\left|G_{J_0-1}(z)-A\right|^2 \\
& \quad=\frac{1}{2Q^2} \sum_{z \in \Gamma_Q}\left(\sum_{r(\bmod 3)}\left|h_{J_0, r}(z)\right|^4-\sum_{r\left(\bmod 3^{J_0}\right)}\left|F_{1, r}(z)\right|^4\right) \\
& \quad+\frac{1}{2Q^2} \sum_{z \in \Gamma_Q} \sum_{k=2}^{J_0-1} \sum_{r\left(\bmod 3^{J_0+1-k}\right)}\left|R_{k, r}(z)\right|^4,
\end{aligned}
\end{equation}
where we used disjointness in the physical variables in the form
\begin{equation*}
|h_{k, r}|^4=|F_{k, r}|^4+|R_{k, r}|^4.
\end{equation*}
Recall in (3) in Lemma \ref{lem:pruning-structure} we obtained that $|h_{k, r}|\le 3\lambda$. This, together with \eqref{eq:quartic-telescoping}, implies that 
\begin{equation*}
\begin{aligned}
& \frac{1}{Q^2} \sum_{z \in \Gamma_Q}\left|G_{J_0-1}(z)-A\right|^2 \\
& \quad \lesim \frac{\lambda^2}{Q^2} \sum_{z \in \Gamma_Q}\left(\sum_{r(\bmod 3)}\left|h_{J_0, r}(z)\right|^2+\sum_{k=2}^{J_0-1} \sum_{
r\ (\mods 3^{J_0+1-k})
}
\left|R_{k, r}(z)\right|^2\right).
\end{aligned}
\end{equation*}
Here we simply ignored the negative term in \eqref{eq:quartic-telescoping}. To continue, we apply \eqref{eq:pruning-projection-energy} and \eqref{eq:discarded-pruning-energy}, and obtain 
\begin{align*}
&\lesim 
\frac{\lambda^2}{Q^2} \sum_{z \in \Gamma_Q}\left(\left|F_{J_0-1}(z)\right|^2+\sum_{k=2}^{J_0-1}\left(\left|F_{k-1}(z)\right|^2-\left|F_k(z)\right|^2\right)\right)\\
&\lesim \frac{\lambda^2}{Q^2} \sum_{z\in \Gamma_Q} |F_1(z)|^2\lesim \lambda^2 A,
\end{align*}
where the last inequality follows from \eqref{eq:pruning-energy-conservation}.  This finishes the proof of (2) in Lemma \ref{lem:key-square-function}.

\subsection{Control of the stopping regions}
The proof of (3) is similar to that of (2). By \eqref{eq:square-function-martingale}, we observe that 
\begin{equation*}
G_k-A= \E_k (G_{J_0-1}-A), \ 0\le k< J_0.
\end{equation*}
Note that on $\Omega_k$ we always have $G_k> 2A$ and therefore  
\begin{equation*}
0\le G_k\le 2(G_k-A).
\end{equation*}
As a consequence, 
\begin{equation*}
\begin{split}
\sum_{z\in \Omega_k} G_k(z)^2
&
\lesim \sum_{z\in \Omega_k} 
|
\E_k(
G_{J_0-1}-A
)(z)
|^2\\
& \lesim 
\sum_{z\in \Omega_k} \E_k(
|G_{J_0-1}-A|^2
)(z),
\end{split}
\end{equation*}
where in the last inequality we applied Cauchy-Schwarz (see also \eqref{eq:conditional-average-properties}). By the fact that $\Omega_k$ is a union of $\Gamma_{k, Q}$-blocks, we further obtain 
\begin{equation*}
\sum_{z\in \Omega_k} G_k(z)^2\lesim 
\sum_{z\in \Omega_k} 
|G_{J_0-1}(z)-A|^2.
\end{equation*}
Now we sum over $k$, and obtain 
\begin{equation*}
\frac{1}{Q^2}
\sum_{k=1}^{J_0-1} \sum_{z\in \Omega_k} G_k(z)^2 \lesim 
\frac{1}{Q^2} \sum_{z\in \Gamma_Q} 
|G_{J_0-1}(z)-A|^2 \lesim \lambda^2 A,
\end{equation*}
where the last inequality follows from (2). This finishes the proof of (3). 

\section{Key Lemma III: Proof of Lemma \ref{lem:key-local-bilinear}}\label{sec:key-local-bilinear}

We close in this section the desired proof of Theorem \ref{thm:weak-l6} by giving the proof of Lemma \ref{lem:key-local-bilinear}.
\begin{proof}[Proof of Lemma \ref{lem:key-local-bilinear}]
Fix $1\le k<J_0$ and set $j:=J_0-k$. Let
$B=z_0+\Gamma_{k,Q}$ be a $\Gamma_{k,Q}$-block, and set
\[
V_i:=\mathcal{P}_{\rho_i}F_k,\qquad i=1,2.
\]
We identify $B$ with $(\Z/3^j\Z)^2$ through
\[
(u,v)\longmapsto
z_0+\left(\frac{u}{3^j},\frac{v}{3^j}\right).
\]
Under this identification, the restriction to $B$ of a Fourier
character with frequency $(\xi,\eta)$ is, up to a constant phase, the
character with frequency $(\xi,\eta)$ modulo $3^j$.

By Lemma \ref{lem:pruning-structure},
\[
\mathcal{P}_{\rho_i}F_k
=
\sum_{\substack{
r\ (\mods 3^j)\\
r\equiv\rho_i\ (\mods 3)
}}
h_{k+1,r}.
\]
Moreover,
$\supp(\widehat{h}_{k+1,r})\subset\Theta_{j,r}$.
Hence, after restriction to $B$, the Fourier support of $V_i$ is
contained in
\[
\mathcal{V}_{j,\rho_i}
:=
\left\{
(\xi,\eta)\in(\Z/3^j\Z)^2:
\xi\equiv\rho_i\ (\mods 3),\
\eta\equiv\xi^2\ (\mods 3^j)
\right\}.
\]

We first record the corresponding bilinear estimate on $B$:
\begin{equation}\label{eq:local-bilinear-estimate}
\sum_{z\in B}|V_1(z)V_2(z)|^2
\le
\frac{1}{\#B}
\left(\sum_{z\in B}|V_1(z)|^2\right)
\left(\sum_{z\in B}|V_2(z)|^2\right).
\end{equation}
Indeed, by Plancherel on $(\Z/3^j\Z)^2$, it suffices to bound the
cardinality of each convolution fiber associated with
$\mathcal{V}_{j,\rho_1}$ and $\mathcal{V}_{j,\rho_2}$.
Fixing the sum frequency $(u,v)$, a possible first frequency
$(\xi,\eta)$ must satisfy
\[
v\equiv \xi^2+(u-\xi)^2\pmod{3^j},
\qquad
\xi\equiv\rho_1\pmod 3,
\qquad
u-\xi\equiv\rho_2\pmod 3.
\]
If $\xi'$ is another possible choice, subtraction gives
\[
2(\xi-\xi')(\xi+\xi'-u)\equiv0\pmod{3^j}.
\]
The second factor is congruent modulo $3$ to
$\rho_1-\rho_2$, and is therefore invertible modulo $3^j$.
Since $2$ is also invertible modulo $3^j$, we conclude that
$\xi\equiv\xi'\pmod{3^j}$. Thus every convolution fiber contains at
most one element, proving \eqref{eq:local-bilinear-estimate}.

It remains to relate the local $L^2$ masses of $V_1,V_2$ to $G_k$.
For distinct $r,r'\pmod{3^j}$, the restrictions of
$h_{k+1,r}$ and $h_{k+1,r'}$ to $B$ have disjoint Fourier supports.
Therefore, by local Plancherel,
\[
\sum_{z\in B}\left(|V_1(z)|^2+|V_2(z)|^2\right)
\le
\sum_{z\in B}
\sum_{r\ (\mods 3^j)}
|h_{k+1,r}(z)|^2
\le
\sum_{z\in B}G_k(z).
\]
Combining this with \eqref{eq:local-bilinear-estimate} and
$ab\le (a+b)^2/4$, we obtain
\[
\sum_{z\in B}|V_1(z)V_2(z)|^2
\le
\frac{1}{4\#B}
\left(\sum_{z\in B}G_k(z)\right)^2.
\]
Finally, $G_k$ is constant on every $\Gamma_{k,Q}$-block, thus
\[
\frac{1}{\#B}
\left(\sum_{z\in B}G_k(z)\right)^2
=
\sum_{z\in B}G_k(z)^2.
\]
Summing up we obtain
\[
\sum_{z\in B}
|\mathcal{P}_{\rho_1}F_k(z)
\mathcal{P}_{\rho_2}F_k(z)|^2
\lesssim
\sum_{z\in B}G_k(z)^2,
\]
which proves \eqref{eq:local-bilinear-square-function}.
\end{proof}

\phantomsection
\addcontentsline{toc}{section}{References}
\bibliographystyle{alpha}
\bibliography{weakrefs}

\end{document}